\documentclass{article}
\usepackage{amsmath, amssymb, amsthm, hyperref, graphicx, enumerate}
\usepackage{pinlabel}
\usepackage{subfigure}  
\usepackage[margin=1.25in]{geometry}

\newtheorem{thm}{Theorem}[section]
\newtheorem{prop}{Proposition}[section]
\newtheorem{lemma}{Lemma}[section]

\theoremstyle{definition}
\newtheorem{defn}{Definition}[section]
\newtheorem{rmk}{Remark}[section]
\newtheorem{question}{Question}[section]

\newcommand\R{\mathbb{R}}

\newcommand{\area}{\operatorname{area}}

\title{Equidistribution of saddle connections on translation surfaces}
\author{Benjamin Dozier \thanks{Department of Mathematics, Stanford University, \href{mailto:benjamin.dozier@gmail.com}{\nolinkurl{benjamin.dozier@gmail.com}}.  Supported in part by  NSF grant DGE-114747.}}
\date{}

\begin{document}
\maketitle

\begin{abstract}
Fix a translation surface $X$, and consider the measures on $X$ coming from averaging the uniform measures on all the saddle connections of length at most $R$.  Then as $R\to\infty$, the weak limit of these measures exists and is equal to the area measure on $X$ coming from the flat metric.  This implies that, on a rational-angled billiard table, the billiard trajectories that start and end at a corner of the table are equidistributed on the table.  We also show that any weak limit of a subsequence of the counting measures on $S^1$ given by the angles of all saddle connections of length at most $R_n$, as $R_n\to\infty$, is in the Lebesgue measure class.  The proof of the equidistribution result uses the angle result, together with the theorem of Kerckhoff-Masur-Smillie that the directional flow on a surface is uniquely ergodic in almost every direction.  
\end{abstract}

\tableofcontents

\section{Introduction}

Consider a polygonal billiard table, with a frictionless billiard ball bouncing around inside according to the rule that, when it hits a side, angle of incidence equals angle of reflection.  This is a much studied dynamical system that exhibits rich and complicated behavior, and it can be seen as a model for basic physical systems, such as a gas molecule bouncing around in a box.  When the ball hits a corner of the table, its future trajectory is not well-defined, and this presence of singular trajectories is in some sense responsible for the complicated behavior of the system.   It is thus of interest to study these singular trajectories, and in particular, to study the finite length ``corner-corner'' trajectories that both start and end at a corner.  

In the case where all the angles of the billiard table are \emph{rational}, many special techniques are available to study the system.   There is an unfolding construction due to Fox-Kershner \cite{fk1936} and Katok-Zemljakov \cite{kz1975}, which involves taking many copies of the original rational billiard table and assembling them into a closed surface with a flat structure, called a \emph{translation surface}.  The billiard trajectories on the original table, each of which consists of many straight segments meeting at angles, correspond to completely straight trajectories on the translation surface.  The flat structure has finitely many singular points, and the straight lines that start and end at singular points, called \emph{saddle connections}, correspond to corner-corner trajectories on the billiard table.

The main result of this paper (Theorem \ref{thm:equidist}) is that the set of saddle connections becomes \emph{equidistributed} on any translation surface.  This implies that, on a rational billiard table, the set of corner-corner trajectories is equidistributed on the table.  A similar result also holds for periodic billiard trajectories (that don't hit corners) on rational billiard tables (see Remark \ref{rmk:cyl}).  

The proof involves translating the question about long saddle connections on an individual translation surface into a question about bounded length saddle connections on varying surfaces; this is an instance of the philosophy of \emph{renormalization}.  The process of varying the surface corresponds to an action of $SL_2(\mathbb{R})$ on the space of translation surfaces.

\subsection{Basic Definitions.} A \emph{translation surface} is a pair $X=(M,\omega)$, where $M$ is a Riemann surface, and $\omega$ is a holomorphic $1$-form.  At a finite set $\Sigma$, the form $\omega$ has zeroes. Away from its zeroes, $\omega$ defines a flat (Euclidean) metric.   The metric has a conical singularity of cone angle $2(n+1)\pi$ at each zero of order $n$.  We will assume throughout that the genus of $M$ is at least $2$, which implies that the set of singular points is non-empty.  When we write $X$, we will often mean the surface endowed with this metric.  The area with respect to the flat metric defines a measure $\lambda$ on $X$, which we will refer to as \emph{area measure}, or simply \emph{area}.   An important class of translation surfaces are those arising from polygonal billiard tables with rational angles via unfolding.  

Our primary objects of study are \emph{saddle connections}, which are geodesic segments that start and end at zeroes (we allow the endpoints to coincide), with no zeroes on the interior of the segment.  For each saddle connection $s$, the integral of the 1-form $\omega$ along $s$ gives an element of $\mathbb{C}$, the \emph{holonomy} vector, which carries the information of the length $|s|$ of $s$ as well as its direction $\operatorname{angle}(s)$ (relative to the horizontal direction).  

We can also consider closed loops not hitting zeroes that are geodesic with respect to the flat metric.  Whenever there is one of these, there will always be a continuous family of parallel closed geodesic loops with the same length.  We refer to a maximal such family as a \emph{cylinder}.  Every cylinder is bounded by a union of saddle connections parallel to the cylinder.  

The bundle $\Omega \mathcal{M}_g$ of holomorphic 1-forms over $\mathcal{M}_g$ (the moduli space of genus $g$ Riemann surfaces), with zero section removed, can be thought of as the moduli space of translation surfaces.  This bundle breaks up into strata of translation surfaces that have the same multiplicities of the zeroes of $\omega$.  We denote by $\mathcal{H}(m_1,\dots,m_k)$ the stratum of surfaces of area $1$ with $k$ zeroes of order $m_1,\ldots,m_k$. 

For more information about the theory of translation surfaces, including the connection to rational billiards, many survey articles are available, such as \cite{zorich2006} and \cite{wright2015}.  
\subsection{Main results.}

The first result concerns equidistribution of saddle connections on the surface.  

\begin{thm}
Given a saddle connection $s$, let $\mu_s$ be the probability measure on $X$ that is uniform on $s$ (i.e. the measure of any subset is proportional to the linear measure of the intersection with $s$).  Let $\mu_R$ be the average of the $\mu_s$ over all $s$ of length at most $R$.  Then as $R\to\infty$, $\mu_R$ converges weakly to area measure on $X$. 
\label{thm:equidist}
\end{thm}

In particular, the result applies to any translation surface arising from unfolding of a polygonal billiard table with all angles rational.  The result implies that, on such a table, if a player takes the shortest shot that starts and ends at a corner, then the next shortest, and so on, the table will get worn evenly (asymptotically).

\begin{figure}[h]
\begin{center}
  \includegraphics[scale=0.7]{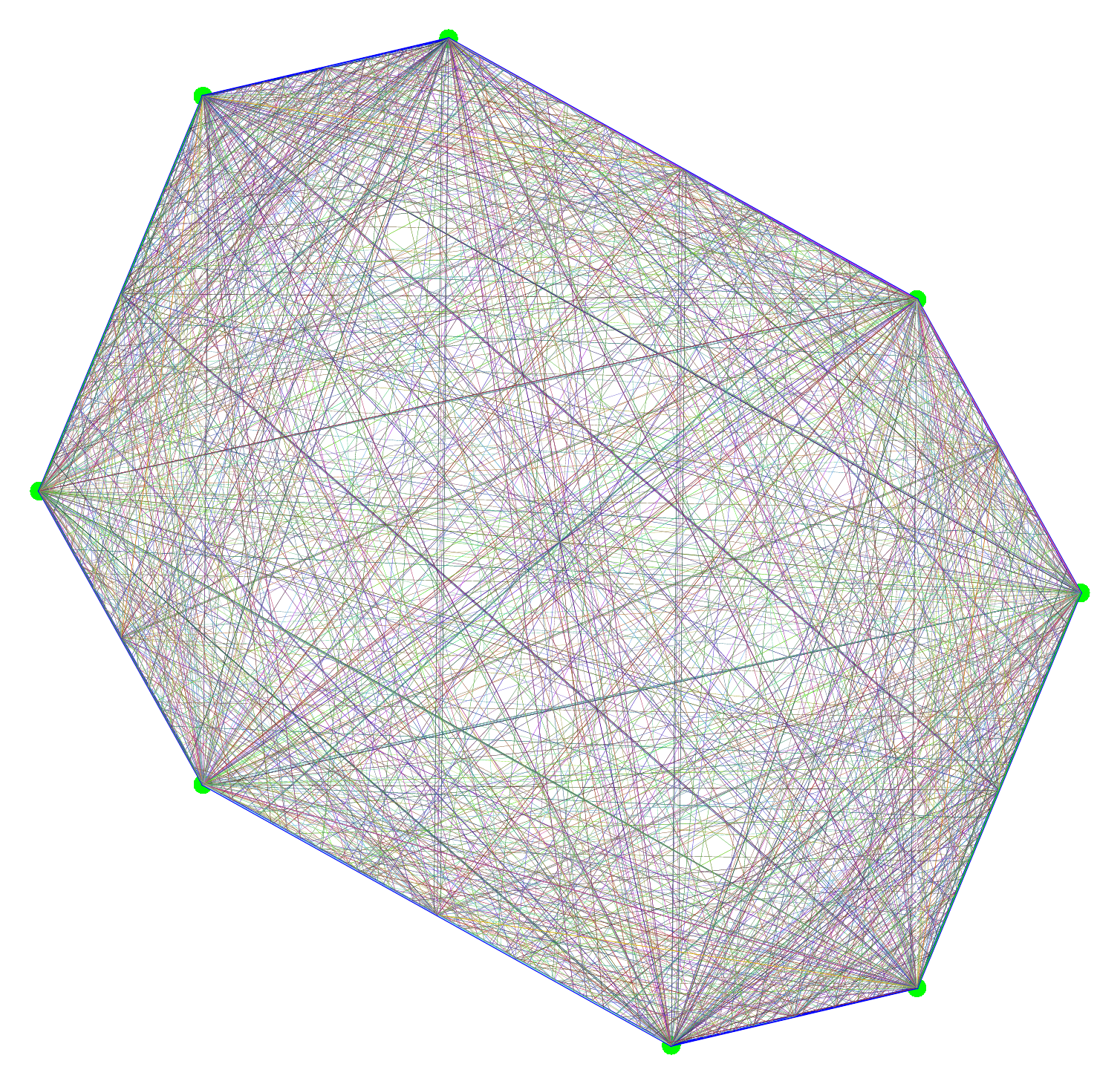}
  \caption{Saddle connections of length at most $R=7$ on a genus two translation surface (opposite sides are identified), in units where the height of the figure is approximately 2.  The thickness of each saddle connection is drawn inversely proportional to its length (so the total amount of ``paint'' used to draw a saddle connection is independent of its length).  This choice of thickness is meant to represent the measures $\mu_s$, which are all probability measures, in Theorem \ref{thm:equidist}.  That theorem says that, as the length bound $R$ goes to infinity, the picture will be uniformly colored. This picture was generated with the help of Ronen Mukamel's $\texttt{triangulated\_surfaces}$ SAGE package.}
  \label{fig:sc_plot}
\end{center}
\end{figure}

\begin{rmk}
  Note that the measures $\mu_s$ corresponding to individual saddle connections do not necessarily become equidistributed as $|s|\to\infty$.  For instance, construct a surface that has a cylinder whose closure is a proper subset of the surface, take a saddle connection contained in the cylinder crossing from boundary to boundary, and then take Dehn twists of this saddle connection about a circumference curve of the cylinder.  See Figure \ref{fig:cyl}.  This gives a family of longer and longer saddle connections, but they all live in the cylinder, so the corresponding measures $\mu_s$ will give zero mass to the complement.
\end{rmk}

\begin{figure}[h]
\begin{center}
  \includegraphics[scale=0.5]{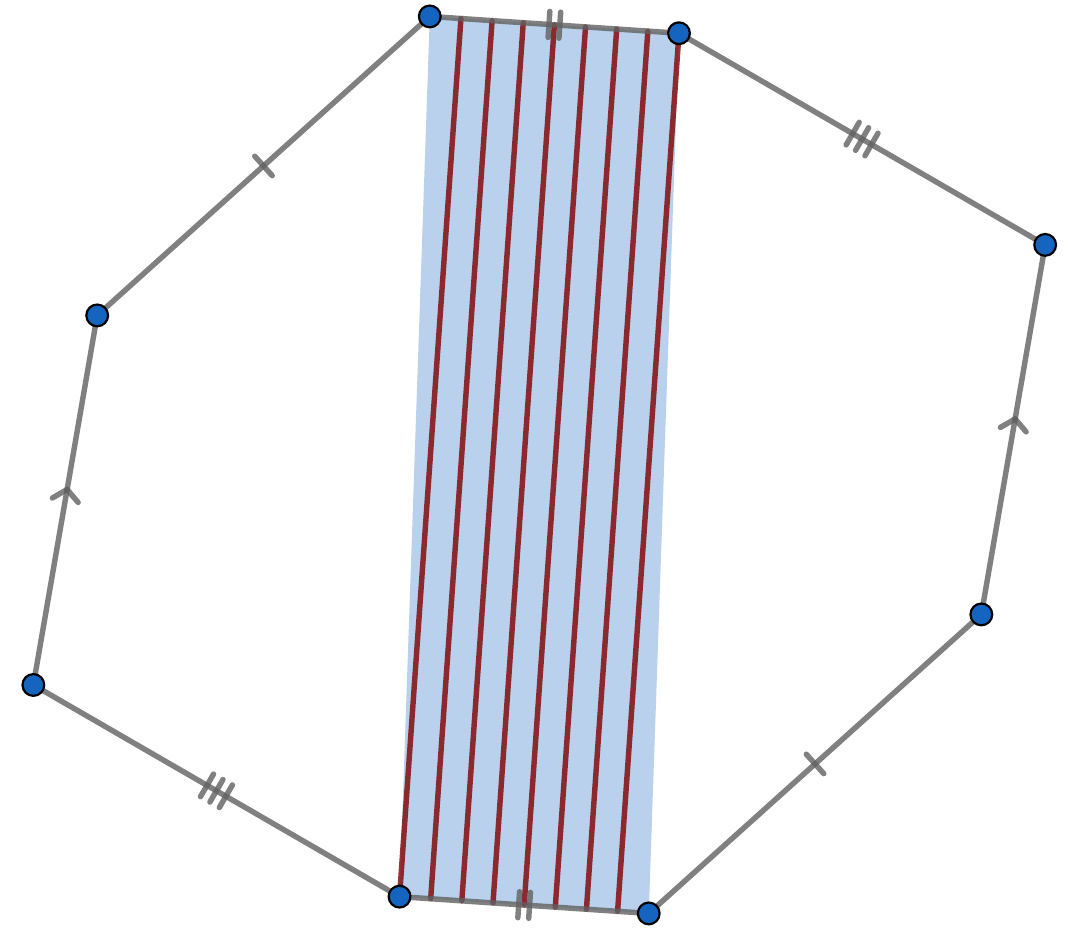}
  \caption{Opposite sides of the polygon are identified to give a genus two translation surface.  A cylinder is shown, together with a long saddle connection contained in that cylinder.}
  \label{fig:cyl}
\end{center}
\end{figure}

Theorem \ref{thm:equidist} can be generalized to any quadratically growing collection of saddle connections.  

\begin{thm}
  Let $S$ be a subset of the set of saddle connections on $X$, let $N_S(R)$ be the number of elements of $S$ of length at most $R$, and let $\mu_{R,S}$ be the average of the uniform measures $\mu_s$ over all $s\in S$ of length at most $R$.  Suppose there is some constant $c$ such that $N_S(R)\ge c R^2$ for all $R$.  Then, as $R\to\infty$, $\mu_{R,S}$ converges weakly to area measure on $X$. 
\label{thm:equidist-subset}
\end{thm}

The proof is essentially the same as that of Theorem \ref{thm:equidist}; we give a few comments on the modifications needed after the proof of Theorem \ref{thm:equidist}.  

\begin{rmk} We also have a version of Theorem \ref{thm:equidist} for cylinders: given a cylinder $c$ on $X$, let $\sigma_c$ be the probability measure on $X$ that is proportional to the restriction of the area measure to $c$.   Let $\sigma_R$ be the average of the $\sigma_c$ over all $c$ of length at most $R$.  Then, as $R\to\infty$, $\sigma_R$ converges weakly to area measure on $X$. This immediately gives a new proof that the set of periodic geodesics on $X$ is a dense subset of $X$, which was first proved by Boshernitzan-Galperin-Kr\"uger-Troubetzkoy \cite{bgkt1998} (for surfaces coming from unfolding billiards on rational polygons), using quite different methods. 
We could also take $\sigma_c$ to be any probability measure on $X$ that is supported on and uniform over some individual periodic geodesic that is part of the cylinder, then take $\sigma_R$ to be the average of the $\sigma_c$ over all $c$ corresponding to periodic geodesics of length at most $R$, and conclude that $\sigma_R$ converges weakly to area measure on $X$. 
The proofs of these results are almost exactly the same as that of Theorem \ref{thm:equidist}, so we will only give the proof of Theorem \ref{thm:equidist} (one also needs as input Masur's quadratic lower bound for cylinders, Theorems 1 and 2 in Masur \cite{masur1988}).  
\label{rmk:cyl}
\end{rmk}

One can also ask whether the \emph{angles} of saddle connections on every surface $X$ become equidistributed on $S^1$ as the saddle connections become longer and longer.  We do not resolve this question, but the theorem below says that in a certain sense the distribution of the angles can't be too far from the Lebesgue measure on the circle.  This result, together with a result of Kerckhoff-Masur-Smille (\cite{kms1986}) that the directional flow on a surface is uniquely ergodic in almost every direction, will be the key ingredients in the proof of Theorem \ref{thm:equidist}.  

\begin{thm}
  Let $X$ be a translation surface, and let $\nu_R$ be the probability measure on $S^1$ given by normalized counting measure on the angles of saddle connections of length at most $R$ on $X$.  Let $\nu$ be a weak limit of a subsequence $\{\nu_{R_n}\}$, $R_n \to \infty$.  Then $\nu$ is in the measure class of Lebesgue measure on $S^1$.  
\label{thm:angle-lebesgue}
\end{thm}

Proving this theorem involves showing lower and upper bounds on the number of saddle connections whose angle lies in a given interval; these bounds are the content of Theorem \ref{thm:lower-bound} and Theorem \ref{thm:upper-bound} below.  Central to the proof is the ``system of integral inequalities'' approach pioneered by Eskin-Margulis-Mozes (\cite{emm1998}) in the lattice context, and brought to the translation surfaces context by Eskin-Masur (\cite{em2001}).  

\begin{rmk}
Using the ``Equidistribution for sectors'' result of Eskin-Mirzakhani-Mohammadi (Theorem 2.6 of \cite{emm2015}), one can show that, assuming the limit of $\nu_R$ exists as $R\to\infty$, then it must be Lebesgue measure.  But it is not known whether the limit exists in general.  
\end{rmk}


\begin{rmk}
  All of the results above also hold with a pair $X=(M,q)$, where $q$ is a holomorphic \emph{quadratic} differential (also known as a \emph{half-translation surface}).  This can be seen by taking the orientation double cover, on which $q$ lifts to a square of a holomorphic 1-form, and applying the results above here.  One can also directly modify the proofs to work for quadratic differentials without much difficulty.  
\end{rmk}

\subsection{Upper and lower bounds.}

We will derive Theorem \ref{thm:angle-lebesgue} from the following more precise upper and lower bounds. 

 Let $N(X,R,I)$ be the number of saddle connections on $X$ of length at most $R$ whose holonomy angle lies in the interval $I \subset S^1$, and let $N(X,R):=N(X,R,[0,2\pi])$ be the total number of saddle connections of length at most $R$.

 \begin{thm}[Upper bound]
  Given $\mathcal{H}$, there exists a constant $c_2$ such that for any $X\in\mathcal{H}$ and interval $I\subset S^1$, there exists a constant $R_0(X,I)$ such that for all $R>R_0(X,I)$,   
$$N(X,R,I) \le c_2 \cdot |I| \cdot  R^2.$$  
\label{thm:upper-bound}
\end{thm}

\begin{thm}[Lower bound]
  Given $\mathcal{H}$, there exists a constant $c_1>0$, such that for any $X\in \mathcal{H}$ and any interval $I\subset S^1$, there exists $R_0(X,I)$ such that for all $R>R_0(X,I)$,   
$$N(X,R,I) \ge c_1 \cdot |I| \cdot R^2. $$
\label{thm:lower-bound}
\end{thm}

A somewhat weaker version of the lower bound, where the constant $c_1$ can depend on the surface $X$, is proved by Marchese-Trevi{\~n}o-Weil (\cite{mtw2016} Theorem 1.9, item (4) and Proposition 4.5).

\begin{proof}[Proof of Theorem \ref{thm:angle-lebesgue} (assuming Theorem \ref{thm:lower-bound} and Theorem \ref{thm:upper-bound})]
  We need to show that the sets of measure zero are the same with respect to $\nu$ and the Lebesgue measure $\lambda$. 

Let $I$ be any interval in $S^1$.  Then, applying Theorem \ref{thm:upper-bound} to the numerator and Theorem \ref{thm:lower-bound} to the denominator in the limit below, we get 
$$\nu(I) = \lim_{n\to\infty}\nu_{R_n}(I) = \lim_{n\to\infty} \frac{N(X,R_n,I)}{N(X,R_n)} \le \frac{c_2 |I|}{2\pi c_1X},$$
i.e. the $\nu$ measure of any interval is at least some fixed constant multiple of the Lebesgue measure.  Since intervals generate the Borel $\sigma$-algebra, we have that $\nu(E) \ge \frac{c_2 |E|}{2\pi c_1}$ for all Borel sets $E$; in particular if $\nu(E)=0$, then $\lambda(E)=0$.   

On the other hand, applying Theorem \ref{thm:lower-bound} to the numerator and Theorem \ref{thm:upper-bound} to the denominator, gives 
$$\nu(I) = \lim_{n\to\infty} \frac{N(X,R_n,I)}{N(X,R_n)} \ge \frac{c_1 |I|}{2\pi c_2},$$
and so we see that, for any Borel set $E$, if $\lambda(E)=0$, then $\nu(E)=0$, which completes the proof.  In fact we have shown that not only do the Radon-Nikodym derivatives of each measure with respect to the other exist, they are also bounded away from infinity and from $0$, with the bounds depending only on the stratum $\mathcal{H}$.   

\end{proof}

\subsection{Analogy with hyperbolic surfaces.}  
Part of the motivation for this work was to prove analogs of certain classical equidistribution results for geodesics on a closed hyperbolic surface $Y$.  In particular, an analog of the version of Theorem \ref{thm:equidist} for periodic geodesics (see Remark \ref{rmk:cyl})  holds for hyperbolic surfaces: the probability measures on $Y$ coming from averaging the uniform measures over all periodic geodesics on $Y$ of length at most $R$ tend to the hyperbolic area measure on $Y$, as $R\to\infty$ (\cite{bowen1972}).  In fact, a stronger result holds: the measures on the unit tangent bundle $T_1Y$ converge to the Liouville measure on $T_1Y$ (the analogous question for translation surfaces is Question \ref{question:tangent-conv}).

In this paper we consider saddle connections and cylinders of periodic geodesics.  One could also consider \emph{flat closed geodesics}: for every homotopy class of closed loop on $X$, we can tighten the loop to some flat closed geodesic, which will be either (i) a periodic geodesic that is part of a cylinder (in which case the geodesic representative is non-unique, since any geodesic in the cylinder is also a representative), or (ii) a union of saddle connections.  These flat closed geodesics are \emph{not} expected to be equidistributed on the surface with respect to Lebesgue measure.  



\subsection{Previous work.}   The study of saddle connections began in earnest with the work of Masur (\cite{masur1988} and \cite{masur1990}), who showed that, for a fixed $X$, there are quadratic upper and lower bounds for the growth of $N(X,R)$ in terms of $R$.  Veech (\cite{veech1989}) showed that in fact there is an exact quadratic asymptotic (i.e. the lower and upper bounds coincide) for a certain class of surfaces having large symmetry groups, now known as Veech surfaces. Later, Veech exploited the analogy between a stratum of translation surfaces $X$ and the space of lattices $SL_n(\mathbb{R})/SL_n(\mathbb{Z})$, to show that there is a constant, the Siegel-Veech constant, that governs the number of saddle connections of length at most $R$, averaged over all surfaces in the stratum (see \cite{veech1998}).  Eskin-Masur (\cite{em2001}, Theorem 1.2) showed that \emph{almost every} surface in 
the stratum has exact quadratic growth asymptotics for saddle connections, with the constant given by the Siegel-Veech constant of the stratum.  One of the key inputs is an ergodic theorem of Nevo (\cite{nevo2017}) that gives equidistribution of large circles in $\mathcal{H}$ centered at almost every point.  

Vorobets showed that for \emph{almost every} $X$ in a stratum $\mathcal{H}$, the angles of saddle connections become equidistributed i.e. the measures $\nu_R$ defined in Theorem \ref{thm:angle-lebesgue} converge to Lebesgue measure on $S^1$ (\cite{vorobets2005}, Theorem 1.9).

\subsection{Outline.}
\begin{itemize}
\item In Section \ref{sec:upper-bound} on upper bounds, we state Proposition \ref{prop:circle-decay}, a key technical tool related to recurrence of pieces of $SL_2(\mathbb{R})$ orbits, and then derive Theorem \ref{thm:upper-bound} (Upper bound) assuming this proposition, following the Eskin-Masur approach.  

\item In Section \ref{sec:lower-bound}, we derive Theorem \ref{thm:lower-bound} (Lower bound).  We follow a strategy pioneered by Masur, and the proof uses Theorem \ref{thm:upper-bound} (Upper bound) as a blackbox.  This section is otherwise independent of Section \ref{sec:upper-bound} and later sections.  

\item Section \ref{sec:proof-cor} derives Theorem \ref{thm:equidist} (Equidistribution) from Theorem \ref{thm:angle-lebesgue} on angle measures, using disintegration of measure and the result of Kerckhoff-Masur-Smillie.  

\item Finally in Section \ref{sec:system-inequalities} we prove the technical Proposition \ref{prop:circle-decay} using the ``system of integral inequalities'' approach.  

\item Section \ref{sec:open-questions} presents some open questions related to the theorems we prove.

\end{itemize}  

\subsection{Acknowledgments.} 
I am very grateful to Maryam Mirzakhani, my thesis advisor, for guiding me with numerous stimulating conversations and suggestions. I would also like to thank Alex Wright for many helpful discussions and detailed feedback.  Finally, I would like to thank the anonymous referee for many helpful comments that improved the exposition.  

\section{Proof of Theorem \ref{thm:upper-bound} (Upper bound)}
\label{sec:upper-bound}

We will use the Eskin-Masur counting strategy  (see \cite{em2001}; an overview of the main ideas is given in \cite{eskin2006}).  We do everything with arcs instead of the whole circle.  Even though we are interested in saddle connections on an individual surface, we will end up studying the moduli space of translation surfaces.  The guiding philosophy is to translate the question about saddle connections of growing length on a fixed translation surface into a question about saddle connections of bounded length on a varying family of surfaces (this idea is often referred to as a type of ``renormalization'').  

\paragraph{Action of $SL_2(\mathbb{R})$.}

There is an action of $SL_2(\mathbb{R})$ on each stratum $\mathcal{H}$ of translation surfaces which will play a central role in our discussion.  To see the action, we first observe that by cutting along saddle connections, we can represent every translation surface as a set of polygons in the plane, such that every side is paired up with a parallel side of equal length.  Since $SL_2(\mathbb{R})$ acts on polygons in the plane, preserving the property of a pair of sides being parallel and equal length, the group acts on the space of translation surfaces.  We will work mostly with elements of the following form:
\begin{eqnarray*}
  g_t =
  \left(\begin{matrix}
    e^{-t} & 0 \\
    0 & e^t
  \end{matrix}\right), \text{  \ \    } r_{\theta} = \left(
        \begin{matrix}
          \cos \theta & -\sin\theta \\
          \sin \theta & \cos\theta
        \end{matrix}\right).
\end{eqnarray*}
The $g_t$ generate Teichm\"uller geodesics, while the action of $r_{\theta}$ rotates the vertical direction by angle $\theta$.  

We first state a proposition that is the key technical tool, and whose proof (given in Section \ref{sec:system-inequalities}) will take up most of the paper. 

Let $\ell(X)$ denote the length of the shortest saddle connection on $X$. 

\begin{prop}
Fix $\mathcal{H}$, and $0<\delta<1/2$.  There is a function $\alpha: \mathcal{H} \to \mathbb{R}$ and a constant $b$ such that for any interval $I\subset S^1$ there exists a constant $c_{I}$ such that for any $X \in \mathcal{H}$, 
$$\int_I \frac{1}{\ell(g_Tr_{\theta}X)^{1+\delta}} d\theta < c_I\cdot e^{-(1-2\delta) T}\alpha(X) + b\cdot |I|,$$
for all $T\ge 0$.  
\label{prop:circle-decay}
\end{prop}

Some related results appear in \cite{em2001} (Theorem 5.2 and Proposition 7.2) and \cite{athreya2006}.  The new elements here are (i) integration over a fixed interval of angles, rather than the whole interval or a changing interval, and (ii) the resulting additive term $b\cdot |I|$.   This proposition is also the key tool needed to prove a result on convergence of Siegel-Veech constants in \cite{dozier2018}; for that application it is essential that the constant $b$ does not depend on the surface $X$.   For all applications, it is crucial that the function $\alpha$ does not depend on $T$.   Section \ref{sec:system-inequalities} is devoted to proof of this proposition, and a generalization, using the ``system of integral inequalities'' approach.  

The proposition should be thought of as a recurrence result for arcs of $g_tr_{\theta}$ ``circles''.  It says that such arcs centered about a point do not have too much mass in the ``thin'' part of the stratum (the part where there is a short saddle connection).  If the center point is itself deep in the thin part, then the radius of the circle containing the arc needs to be taken to be large.  


We now state several intermediate lemmas that we use to execute the Eskin-Masur counting strategy.  

\begin{lemma}
  For any surface $X\in \mathcal{H}$ and interval $I\subset S^1$, we can find an interval $I'\subset S^1$ with $|I'| = |I|$, such that 
  $$N(X,R,I) -N(X,R/2,I) \le 2 R^2\int_{I'} N(g_{\log R}r_{\theta}X,2) d\theta.$$
\label{lemma:trapezoid}
\end{lemma}

This is a modification of Proposition 3.5 in \cite{em2001}, which deals with the case where $I$ is the whole circle.  

\begin{proof}
  Consider the ``annular wedge'' region 
$$W = \{(r\cos \theta, r \sin \theta) : 1/2 \le r \le 1, |\theta| \le \pi/4 \} \subset \mathbb{R}^2.$$
Let $f: \mathbb{R}^2 \to \mathbb{R}$ be the indicator function of this region.  See Figure \ref{fig:trapezoid}.  

\begin{figure}[h]
\begin{center}
  \includegraphics[scale=2.5]{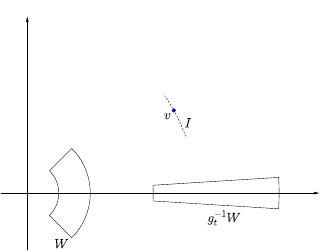}
  \caption{Regions used in proof of Lemma \ref{lemma:trapezoid}}
  \label{fig:trapezoid}
\end{center}
\end{figure}

The idea of the proof is to apply $g_t^{-1}$ to this region, which makes it long and thin, and to rotate using $r_{\theta}$, which will allow us to count vectors with length in a certain range and angle in $I$.  

Let $2I$ be the interval with the same center as $I$ and twice the length, and let $t=\log R$.  We then see that for any $v\in \mathbb{R}^2$, and 
\begin{align*} \int_{2I} f(g_tr_{\theta}v) d\theta \ge
\begin{cases}
  e^{-2t} &\text{ if } e^t/2 \le \|v \| \le e^t \text{ and } \operatorname{angle}(v) \in I \\
  0 &\text{ otherwise.}
\end{cases}
\end{align*}
The above holds because $f(g_tr_{\theta}v) \ne 0 \Leftrightarrow g_tr_{\theta}v \in W \Leftrightarrow r_{\theta}v \in g_t^{-1}W,$ and if this holds for some $\theta\in I$, then it holds for a range of angles of size approximately $2e^{-2t}$, which is (approximately) the angle subtended by $g_t^{-1}W$, as seen from the origin. We use the weaker inequality that the range of angles has size at least $e^{-2t}$, for sufficiently large $t$.  
Note that we use $2I$ instead of $I$ because of edge effects when $\operatorname{angle}(v)$ is close to the boundary of $I$. 

Now summing over the set $\Lambda(X)$ of saddle connection periods for $X$ gives
\begin{align*}
  N(X,I,e^t) - N(X,I,e^t/2) & \le  e^{2t} \sum_{v\in \Lambda(X)} \int_{2I} f(g_tr_{\theta}v) d\theta \\
  & =  e^{2t} \int_{2I} \left(\sum_{v\in \Lambda(X)} f(g_tr_{\theta}v)\right)d\theta \\
  & \le  e^{2t} \int_{2I} N(g_tr_{\theta}X,2) d\theta,
\end{align*}
where in the last inequality we have used the fact that the region $W$ is contained in the ball of radius $2$ centered at the origin.  
Now split $2I$ into two equal length intervals, and let $I'$ be the one on which the integral of $N(g_tr_{\theta}X,2)$ is larger.  Then we get 
$$N(X,I,e^t) - N(X,I,e^t/2) \le 2 e^{2t} \int_{I'} N(g_tr_{\theta}X,2),$$
which is the desired result.

\end{proof}

\begin{lemma}
  For any $\mathcal{H}$, $R>0$, and $\delta>0$, there exists a constant $C$ such that for all $X\in \mathcal{H}$, 
  $$N(X,R) \le \frac{C}{\ell(X)^{1+\delta}},$$
  where $\ell(X)$ denotes the length of the shortest saddle connection on $X$.  
\label{lemma:short_sc}
\end{lemma}

The above appears as Theorem 5.1 in \cite{em2001}.  We will not reprove it here - the proof also uses the system of integral inequalities and involves further technical difficulties.  

We are now ready to prove the upper bound.  

\begin{proof}[Proof of Theorem \ref{thm:upper-bound}]

  Fix $X\in \mathcal{H}$, $\delta>0$, and an interval $I\subset S^1$.  Applying Lemma~\ref{lemma:trapezoid}, Lemma \ref{lemma:short_sc}, and then Proposition \ref{prop:circle-decay}, we get that 
  there is an interval $I'$ with $|I'|=|I|$ such that
\begin{align*}
  N(X,R,I) -N(X,R/2,I) &\le 2 R^2\int_{I'} N(g_{\log R}r_{\theta}X,2) d\theta \\
  &\le 2R^2 \int_{I'} \frac{C}{\ell(g_{\log R}r_{\theta}X)^{1+\delta}} d\theta \\
  &\le 2 C \cdot  R^2 \left( c_{I'} \cdot e^{-(1-2\delta)\log R} \alpha(X) + b\cdot |I|\right).
\end{align*}
For $R$ sufficiently large (potentially depending on all the other parameters), we have
$$N(X,R,I) - N(X,R/2,I) \le 4Cb \cdot R^2 \cdot |I|. $$
Note that the constant $4Cb$ does not depend on $I$.  We then get the desired estimate for $N(X,R,I)$ by an easy geometric series argument.  
\end{proof}

\section{Proof of Theorem \ref{thm:lower-bound} (Lower bound)}
\label{sec:lower-bound}

A somewhat weaker statement than Theorem \ref{thm:lower-bound}, namely a quadratic lower bound for which the constant $c_1$ depends on the surface $X$, was proven by Marchese-Trevi{\~n}o-Weil (\cite{mtw2016}, Theorem 1.9, item (4) and Proposition 4.5; they make use of the main results of \cite{chaika2011}).  One can also prove this weaker version by closely following the original strategy for proving lower bounds (without restriction on angle) developed by Masur in \cite{masur1988} (Sections $\S$1 and $\S$2).

The proof below, which gives a constant depending only on the stratum,  is a hybrid of the techniques of Marchese-Trevi{\~no}-Weil and those of Masur.  The first part, following Marchese-Trevi{\~no}-Weil, is to prove a Ces\`aro type lower bound for saddle connections, Lemma \ref{lemma:cesaro-lower} below.  The saddle connections are produced in Lemma \ref{lemma:near-horiz} by considering surfaces deformed by $g_t$, and using the fact that there is a universal upper bound on the length of the shortest saddle connection on a translation surface.  Then, following Masur, we combine the Ces\`aro type lower bound with the quadratic upper bound of Theorem \ref{thm:upper-bound} to get the desired lower bound (the quadratic upper bound is used as a blackbox; nothing in its proof is needed).

We will use the following basic, well-known fact, which gives a universal upper bound on the length of the shortest saddle connection. 

\begin{lemma}
  For any unit-area translation surface $X$ in any stratum $\mathcal{H}$, we have the following universal bound on the length of the shortest saddle connection $\ell(X)$:
  $$\ell(X) \le \frac{2}{\sqrt{\pi}}.$$
    \label{lemma:systole}
\end{lemma}

\begin{proof}
  Pick a singular point $q\in \Sigma$, and consider the set $B_R(q)$ of points of distance at most $R$ from $q$ in the flat metric.  For small $R$, this gives an embedded topological disk that does not contain any other singular point besides $q$.  As we increase $R$, the ball eventually either (i) hits a different singular point $p\in \Sigma$, or (ii) the boundary circle intersects itself.  Let $R_0$ be the smallest value for which one of these two events occurs.  Now if $B_R(q)$ is embedded, then $\operatorname{Area}(B_R(q)) > \pi R^2$ (the area grows like some integral multiple of the area of a ball in the Euclidean plane, because of the cone angle at the center).  On the other hand, if it is embedded, then $\operatorname{Area}(B_R(q))\le \operatorname{Area}(X) =1$.  Hence $1\ge \pi R_0^2$, and so $R_0 \le \frac{1}{\sqrt\pi}$.

  If (i) occurred, then the geodesic segment through $B_R(q)$ connecting $q$ to $p$ is a saddle connection of length $R_0$.  Now suppose (ii) occurs, and $z$ is the point at which the boundary circle intersects itself, which we can assume is not a singular point.  There are two geodesic segments $m_1,m_2$, coming from radii of the disk, that connect $q$ to $z$.  Now if $m_1$ and $m_2$ have the same angle on the surface, then the concatenation $\gamma$ of $m_1$ and $m_2$ is a saddle connection of length $2R_0$.  If $m_1,m_2$ have different angle, then their holonomies (the integral of the 1-form $\omega$ over the path) are not real multiples of each other.  The holonomy of the concatenation $\gamma$ is the sum of the holonomies of $m_1$ and $m_2$, and so is non-zero, which implies that $\gamma$ is not a null-homotopic loop, since homotopy preserves the holonomy.  The length of $\gamma$ is $2R_0$.  Now $\gamma$ can be tightened to have minimal length in its homotopy class - the resulting loop will be a concatenation of saddle connections.  So there will be some saddle connection in the concatenation with length at most $2R_0$.

  In all of the cases we get a saddle connection of length at most $2R_0 \le \frac{2}{\sqrt{\pi}}$.  
\end{proof}

The following result appears in \cite{mtw2016} (Proposition 4.1, item (1)).  

\begin{lemma}
  There exists a constant $C$ such that for any $X$ in any stratum $\mathcal{H}$ and any $R$ sufficiently large (depending on $X$), there is a saddle connection $s$ on $X$ such that angle $\theta$ between $s$ and the horizontal direction satisfies
  $$|\theta| \le \frac{C}{R |s|}.$$
  \label{lemma:near-horiz}
\end{lemma}

\begin{proof}
  The idea is to deform the surface using the diagonal matrix $g_t$, take a saddle connection on the surface $g_tX$ of uniformly bounded length, and the consider the corresponding saddle connection on $X$, which will be nearly horizontal.

  Let $s'$ be the shortest saddle connection on $g_tX$, and let $s$ be the saddle connection on $X$ that is the image of $s'$ under $g_t^{-1}$.  Let $s'_x,s'_y$  be the $x$ and $y$ components of the holonomy vector of $s'$.  The components of the holonomy vector of $s$ are $e^{t}s_x'$ and $e^{-t} s_y'$.  Since the set of saddle connection holonomies is invariant under $v\mapsto -v$, we can assume that $s_x'\ge 0$.  We will assume that in fact $s'_x>0$ (if $s_x'=0$, then $g_tX$, and hence also $X$, would necessarily have a vertical saddle connection; we can rotate very slightly to get rid of the horizontal saddle connection, and then running the argument with the slightly rotated surface will produce a saddle connection very close to horizontal on the slightly rotated surface, which will correspond to a saddle connection on $X$ that is also very close to horizontal).

  By the definition of $\ell(X)$, we have 
\begin{align}
  \label{eq:len_sys}
  \left(e^{-t}s_y'\right)^2 + \left( e^ts_x' \right)^2 = |s|^2 \ge \ell(X)^2.  
\end{align}

Now we will suppose that $t$ is large so that $e^t \ge \frac{\sqrt{2}}{\ell(X)} \cdot \frac{2}{\sqrt{\pi}}$.  By Lemma \ref{lemma:systole}, $s'$ has length at most $\frac{2}{\sqrt\pi}$, and hence
\begin{align*}
  e^{-t}s_y' \le  e^{-t}|s'| \le e^{-t} \frac{2}{\sqrt\pi} \le \frac{\ell(X)}{\sqrt{2}} \cdot \frac{\sqrt{\pi}}{2} \cdot \frac{2}{\sqrt\pi} = \ell(X)/\sqrt{2},
\end{align*}
hence
\begin{align*}
  \left(e^{-t}s_y' \right)^2 \le \frac{\ell(X)^2}{2}. 
\end{align*}

Together with \eqref{eq:len_sys} this implies that
\begin{align*}
  e^ts_x' \ge e^{-t}s_y' 
\end{align*}
and so we must have
\begin{align*}
  e^ts_x' \ge |s|/\sqrt{2}. 
\end{align*}

Now using this, and the fact that $s_y'\le |s'| \le 2/\sqrt{\pi}$, we can estimate the angle $\theta$ by

\begin{align}
  \label{eq:theta} 
   |\theta| \le \tan |\theta| =\frac{e^{-t}s_y'}{e^ts_x'} \le \frac{e^{-t}\cdot 2/\sqrt{\pi}}{|s|/\sqrt{2}}. 
\end{align}

Now for $R \ge \frac{\sqrt{2} (2/\sqrt{\pi} )^2}{\ell(X)}$, choose $t$ such that $e^t = \frac{R}{2/\sqrt{\pi}} \ge  \frac{\sqrt{2} (2/\sqrt{\pi})}{\ell(X)}$.  For such $t$, the condition on $e^t$ assumed above is satisfied.   Continuing from \eqref{eq:theta}, we get 

$$ |\theta| \le \frac{e^{-t}\cdot 2/\sqrt{\pi}}{|s|/\sqrt{2}} = \frac{2/\sqrt{\pi}}{R} \cdot \frac{ 2/\sqrt{\pi}}{|s|/\sqrt{2}} = \frac{2\sqrt{2}/\pi}{R|s|},$$

which is the desired result with $C=2\sqrt{2}/\pi$.  
  
\end{proof}

\begin{lemma}
  There exists an absolute constant $c>0$, such that for any $X$ in any stratum $\mathcal{H}$ and any interval $I \subset S^1$, there is some $R_0(X,I)$, such that for $R > R_0(X,I)$,
  $$\sum_{|s|\le R,\ \operatorname{angle}(s)\in I}\frac{1}{|s|} \ge c|I|R,$$
  where the sum is taken over saddle connections $s$ on $X$.  
  \label{lemma:cesaro-lower}
\end{lemma}

\begin{proof}
  Let $I/4$ be the interval with the same center as $I$, and length $|I|/4$.  For $R$ sufficiently large, and for each angle $\alpha\in I/4$, we apply Lemma \ref{lemma:near-horiz} to the rotated surface $r_{\alpha}X$, which gives a saddle connection $s$ with $|s|\le R$ and
  $$|\operatorname{angle}(s)-\alpha| < \frac{C}{R|s|}.$$

  Now for $R$ large enough (depending on $I$),
  $$|\operatorname{angle}(s)-\alpha| < \frac{C}{R|s|} \le |I|/4,$$
  
  and so $\operatorname{angle}(s)$ must be in $I$.

  These facts imply that

$$\bigcup_{|s|\le R,\ \operatorname{angle}(s)\in I}\left\{\phi: |\operatorname{angle}(s)-\phi| < \frac{C}{R|s|}\right\} \supset I/4. $$
 Comparing the lengths of the two sets, it follows that

$$\sum_{|s|\le R,\ \operatorname{angle}(s)\in I} \frac{2C}{R|s|} \ge |I|/4,$$
hence

$$\sum_{|s|\le R,\ \operatorname{angle}(s)\in I} \frac{1}{|s|} \ge \frac{1}{8C}|I| R,$$
which gives the desired result with constant $c=\frac{1}{8C} = \frac{\pi}{16\sqrt{2}}$.  
\end{proof} 

We will combine Lemma \ref{lemma:cesaro-lower} with the quadratic upper bound via the following general lemma about sequences adapted from \cite{masur1988} (Section \S2).  
\begin{lemma}
  Given a sequence $a_1 \le a_2 \le \cdots $ of positive real numbers, let $N_R=|\{i : a_i \le R\}|$, and suppose that there exist positive constants $c_0, \bar C_0,$ such that for all $R$ sufficiently large

  \begin{align}
    \label{eq:cesaro-bound} 
      \sum_{k: a_k \le R} \frac{1}{a_k} \ge c_0 R,
  \end{align}
and
  \begin{align}
    \label{eq:upper-bound}
    N_R \le \bar C_0 R^2.
  \end{align}
  Then for $R$ sufficiently large,
  \begin{align}
    \label{eq:lower-bound}
    N_R\ge c_3R^2,
  \end{align}
  where $c_3 = \frac{c_0^2}{72 \bar C_0}$.
  
\label{lemma:cesaro-quad}
\end{lemma}

\begin{proof}
  The idea is to assume the lower bound fails along a subsequence, use the quadratic upper bound \eqref{eq:upper-bound} to bound from above the first part of the sum $\sum 1/a_k$, for the small $a_k$, and then get a contradiction to \eqref{eq:cesaro-bound} by considering a subsequence where the lower bound \eqref{eq:lower-bound} fails, and using this to bound above the remaining later part of the sum $\sum 1/a_k$.
  
Assume that $R_i\to \infty$ is a sequence such that
\begin{align}
  \label{eq:no-lower}
  N_{R_i} < \frac{c_0^2}{72 \bar C_0 } R_i^2. 
\end{align}
We wish to obtain a contradiction to \eqref{eq:cesaro-bound}; we first bound the first part of the sum in \eqref{eq:cesaro-bound} using the quadratic upper bound \eqref{eq:upper-bound}.  For any positive integer $\ell$, we have
\begin{align}
  \sum_{a_k \le R/2^{\ell}} \frac{1}{a_k} &\le \sum_{j\ge \ell} \frac{N_{R/2^j} - N_{R/2^{j+1}}}{R/2^{j+1}} \\
                                     &\le F + \sum_{j\ge \ell} \frac{\bar C_0(R/2^j)^2}{R/2^{j+1}} \text{ \ \ (by \eqref{eq:upper-bound})}\\
                                     & = F + \sum_{j\ge \ell} \frac{\bar C_0 R}{2^{j-1}} \\
                                     & = F + \frac{\bar C_0 R}{2^{\ell-2}},   \label{eq:first-part}
\end{align}
where $F$ is some constant (independent of $\ell$) that absorbs the smaller $a_k$ that are below the threshold value of $R$ for which the quadratic upper bound \eqref{eq:upper-bound} is guaranteed hold.

Now choose $\ell$ such that
\begin{align}
  \label{eq:choice-ell}
  \frac{\bar C_0}{2^{\ell-2}} < \frac{c_0}{3} \le \frac{\bar C_0}{2^{\ell-3}}.  
\end{align}

Now using the bound from \eqref{eq:first-part} for the first part of the sum in \eqref{eq:cesaro-bound}, and the assumption \eqref{eq:no-lower} to bound the second part gives
\begin{align}
  \label{eq:1}
  \sum_{a_k \le R_i} \frac{1}{a_k} &= \sum_{a_k \le R_i/2^{\ell}}\frac{1}{a_k} + \sum_{R_i/2^{\ell} < a_k \le R_i} \frac{1}{a_k}\\
                                   &\le F+ \frac{\bar C_0 R_i}{2^{{\ell}-2}} + \frac{N_{R_i}}{R_i/2^{\ell}} \text{ \ \ (by \eqref{eq:first-part}, and definition of }N_R)\\
                                   & \le F + \frac{c_0}{3} R_i +\frac{1}{R_i/2^{\ell}} \frac{c_0^2}{72 \bar C_0 } R_i^2 \text{ \ \ (by \eqref{eq:choice-ell} and \eqref{eq:no-lower})} \\
                                   &= F + \frac{c_0}{3} R_i + \frac{c_0^2}{9\bar C_0/2^{{\ell}-3}} R_i \\
                                   &\le F + \frac{c_0}{3} R_i + \frac{c_0^2}{9(c_0/3)} R_i \text{ \ \ (by \eqref{eq:choice-ell})} \\
                                   &= F + \frac{2c_0}{3} R_i,
\end{align}
which contradicts \eqref{eq:cesaro-bound} for large $i$.  

\end{proof}

We are now ready to prove the desired quadratic lower bound.  

\begin{proof}[Proof of Theorem \ref{thm:lower-bound}] 

  We apply Lemma \ref{lemma:cesaro-quad} with $a_k$ the length of the $k^{\text{th}}$ shortest saddle connection $s$ with $\operatorname{angle}(s)\in I$, so that $N_R=N(X,R,I)$. The first hypothesis in Lemma \ref{lemma:cesaro-quad} is  satisfied if we take $c_0=c|I|$, where $c$ is the constant from Lemma \ref{lemma:cesaro-lower}. The second hypothesis is satisfied with $\bar C_0=c_2|I|$, where $c_2$ is the constant from Theorem \ref{thm:upper-bound}.  Then Lemma \ref{lemma:cesaro-quad} gives that
  $$N(X,R,I)  \ge \frac{(c|I|)^2}{72 c_2 |I|} R^2 =\frac{c^2}{72c_2} |I|R^2,$$
for all $R$ sufficiently large (depending on $X$ and $I$).  Neither $c$ nor $c_2$ depend on the surface $X$ (though $c_2$ depends on the stratum $\mathcal{H}$), so we get the desired result.  

\end{proof}

\section{Proof of Theorem \ref{thm:equidist}} 
\label{sec:proof-cor}

In the proof we will work with $X'=X-\Sigma$ and the unit tangent bundle $T_1X'$ (notice that on the tangent spaces over the points of $\Sigma$, the inner product induced by the $1$-form $\omega$ is identically zero, so we cannot sensibly speak of $T_1X$).  Each saddle connection $s$ gives a probability measure $\eta_s$ on $T_1X'$, which is supported on the set of unit tangent vectors to $s$ (excluding the endpoints), and characterized by the property that if we take a measurable subset of these tangent vectors to $s$ whose distance from $\Sigma$ is bounded away from zero, then if we flow in the direction of $s$ for a small amount of time, the $\eta_s$ measure of the image set is the same.  We define probability measures $\eta_R$ by averaging the $\eta_s$ over all $s$ of length at most $R$.  

\paragraph{Idea of proof.} Here is the idea, ignoring the issue that the norm on the tangent bundle to $X$ isn't properly defined at the singular points $\Sigma$, and that there is a complicated set of tangent vectors that flow into $\Sigma$. We argue by contradiction, and look at the corresponding measures $\eta_R$ on the unit tangent bundle - we can find a subsequence that converges to some measure $\eta$ which does not project to area measure on $X$.  Since the converging measures come from longer and longer saddle connections, the limit $\eta$ will be invariant under the geodesic flow.  Now we disintegrate $\eta$ along the angle map that takes a unit tangent vector to its angle in $S^1$; for each $\theta\in S^1$, this gives a measure on the fiber over $\theta$ (the fiber is identified with $X$) that is invariant under the $\theta$ directional flow.  The push-forward of the original measure along the angle map is a measure $\nu$ on $S^1$ that is a limit of angle measures $\nu_{R_n}$, so we can apply Theorem \ref{thm:angle-lebesgue} to show that it is absolutely continuous with respect to Lebesgue measure.  This means that the work of Kerckhoff-Masur-Smillie (\cite{kms1986}) applies: for a.e. $\theta$ (with respect to Lebesgue, and hence to $\nu$), the flow in that direction is uniquely ergodic, and hence the measure on the fiber over $\theta$ must be area measure.  The fact that the fiber measures come from disintegration, and that $\nu$ is the push-forward of $\eta$ along the angle map, means that the push-forward of $\eta$ along the projection to the surface $X$ is an average of the fiber measures with respect to $\nu$.  Since $\nu$-a.e. fiber measure is area measure, $\eta$ is also area measure.  

\subsection{Technical lemmas on saddle connections}
\label{sec:sc}

The goal of this subsection is to establish Lemma \ref{lemma:no-escape}, which says that saddle connections do not concentrate around singular points and will be needed in the proof of Theorem \ref{thm:equidist} below.  We establish Lemma \ref{lemma:no-escape} as a consequence of Lemma \ref{lemma:poly_inter}, which in turn follows from Lemma \ref{lemma:seg_inter}.

Lemma \ref{lemma:seg_inter} says that the number of times that a saddle connection can hit a short segment is small relative to the length of the saddle connection.   Its proof has the most substance, and we give it first.

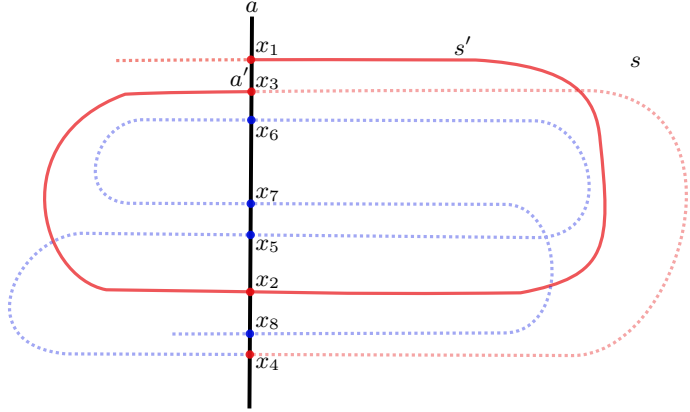
\begin{figure}[ht]
\labellist
\small\hair 2pt
\pinlabel $a$ at 223 303
\pinlabel $s$ at 512 262
\pinlabel $x_1$ at 235 272
\pinlabel $x_2$ at 235 96
\pinlabel $x_3$ at 235 246
\pinlabel $x_4$ at 235 34
\pinlabel $x_5$ at 235 123
\pinlabel $x_6$ at 235 208
\pinlabel $x_7$ at 235 162
\pinlabel $x_8$ at 235 65
\pinlabel $s'$ at 382 274
\pinlabel $a'$ at 214 251
\endlabellist
\centering
  \includegraphics[scale=0.5]{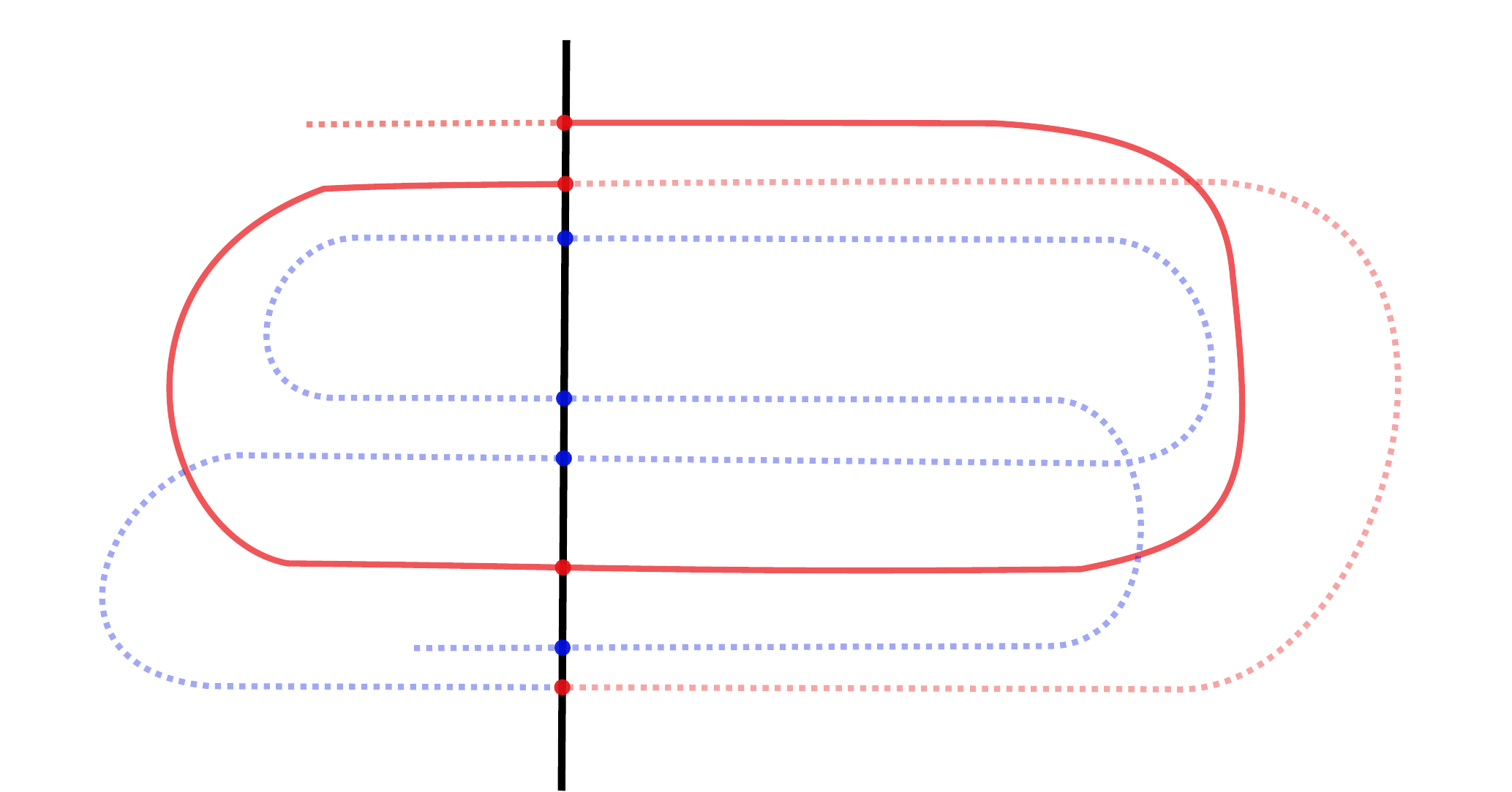}
  \caption{Proof of Lemma \ref{lemma:seg_inter}.  The red points are group $A_1$, while the blue are $A_2$.}
  \label{fig:seg_sc}
\end{figure}

\begin{lemma}
  Let $s$ be a saddle connection on $X$, and let $a$ be a straight segment  (i.e. geodesic with respect to the flat metric) avoiding the singular set $\Sigma$, and which is not parallel to $s$. Then 
$$\frac{\#(s\cap a)}{|s|} \le \left(\frac{|a| + 1}{\ell(X)}\right)^2.$$
\label{lemma:seg_inter}
\end{lemma}
\begin{proof}
  The idea of the proof is that if $s$ intersects $a$ too many times, between some pair of intersection points we could find a short segment of $s$, which together with the segment of $a$ between the two points, would give a short loop that is not homotopically trivial.  Since we need both the part of the loop coming from $a$ and the part from $s$ to be short, the proof is slightly complicated.  

Suppose, for the sake of contradiction, that for some saddle connection $s$, the inequality above fails.  Let $K=\left(\frac{|a|+1}{\ell(X)}\right)^2$.  So there are more than $K\cdot |s|$ points in $s\cap a$; label them $x_1,x_2,\ldots$, in the order in which $s$ (choose some orientation) hits them.  Now divide these into groups $A_i$ of $f(K)$ consecutive points, where $f$ is a function of $K$ to be chosen later.  The total number of groups is (within $1$ of) $\#(s\cap a)/f(K)$.  We can divide $s$ into $\# (s\cap a)/f(K)$ segments $s_i$, where $s_i$ contains all the intersection points in $A_i$.  Now $\sum_i |s_i|=|s|$, so there is some $k$ such that 
$$|s_k| \le \frac{|s|}{\#(s\cap a) /f(K)} < \frac{f(K)}{K}.$$
  The points in $A_k$ all lie on $a$, so by pigeonhole there exist two points $x_i,x_j$ such that the sub-segment $a'$ of $a$ joining $x_i$ to $x_j$ satisfies $|a'| \le |a| /\# A_i = |a|/f(K)$.  Let $s'$ be the sub-segment of $s_k\subset s$ connecting $x_i,x_j$.  

Note that $s'\cup a'$ is a closed loop.  We claim that is not null-homotopic.  In fact, if it were, then it would represent the zero element in the relative homology group $H_1(X,\Sigma; \mathbb{Z})$, which would imply that the period of the $1$-form $\omega$ along $s'$ would equal the negative of the period along $a'$, but this cannot happen because $s',a'$ are straight segments in non-parallel directions.  

On the other hand 
$$|s'\cup a'| = |s'| + |a'| \le |s_k| + \frac{|a|}{f(K)} < \frac{f(K)}{K} + \frac{|a|}{f(K)}.$$

Now we choose $f(K) = \sqrt{K}$, which gives

$$ |s'\cup a| < \frac{1}{\sqrt{K}} + \frac{|a|}{\sqrt{K}} = \frac{1}{\sqrt{K}} (1+|a|)  = \ell(X).$$

This gives a contradiction, since $s'\cup a$ can be homotoped into either (i) a straight periodic orbit avoiding singular points, or (ii) a union of saddle connections.  In both cases, the new path is no longer than the original one.  In case (i) there is a saddle connection in the boundary of the corresponding cylinder of length less than $\ell(X)$, contradiction.  In case (ii) one (in fact any) of the saddle connections in the union has length less than $\ell(X)$, again a contradiction.  

\end{proof}

The next lemma says that the proportion of time a saddle connection spends in certain small polygons is small.  

\begin{lemma}
  Let $S_{\epsilon}$ be the interior of a polygon on $X$, all of whose sides are length $\epsilon$, and all of whose angles are right angles.  Let $s$ be a saddle connection on $X$ that is not parallel to any of the sides of the polygon.  Then, for $\epsilon$ small,
$$\frac{|s\cap S_{\epsilon}|}{|s|} \le \frac{4k\sqrt{2} \cdot \epsilon(\epsilon+1)^2}{\ell(X)^2},$$
where $2\pi k$ is the maximum cone angle at a singular point of $X$. 
\label{lemma:poly_inter}
\end{lemma}

\begin{proof}
  Note that each component of the intersection $s\cap S_{\epsilon}$ has length at most $\epsilon\sqrt{2}$, and the number of components of $s\cap S_{\epsilon}$ is bounded by the sum of the number of intersection points of $s$ with the sides of the polygon $a_1,\ldots,a_n$.  Since the polygon is small, it contains at most one singular point, and the cone angle determines the number of sides $n$, so $n\le 4k$.  Applying Lemma \ref{lemma:seg_inter} gives 
  \begin{align*}
    |s\cap S_{\epsilon}| \le \epsilon \sqrt{2} \sum_{i=1}^n\#(s\cap a_i) \le \epsilon \sqrt{2} \sum_{i=1}^n |s| \left(\frac{|a_i| + 1}{\ell(X)}\right)^2 \le 4k\sqrt{2} \cdot |s| \cdot \frac{\epsilon(\epsilon+1)^2}{\ell(X)^2}, 
  \end{align*}
and dividing through by $|s|$ gives the desired result.  
\end{proof}

We now have what we need to prove Lemma \ref{lemma:no-escape}, which was our goal.  

\begin{lemma}
 Let $R_n$ be any sequence of positive real numbers tending to infinity.  For the probability measures $\{\eta_{R_n}\}_{n=1}^{\infty}$ on $T_1X'$ (defined at the beginning of this section), there exists a subsequence $\{\eta_{R_{a_n}}\}$ that converges weakly to a probability measure $\eta$.  
\label{lemma:no-escape}
\end{lemma}

\begin{proof}
Using standard arguments, we first show that a limit measure exists, and then, using further lemmas below, that the limit is a probability measure.   

To construct the subsequence, began by choosing increasing exhausting compact subsets $K_j\subset T_1X'$, defined as the set of unit tangent vectors over points of $X$ whose distance to a point of $\Sigma$ is at least $1/j$.   By Banach-Alaoglu, for each $j$ we can find a subsequence $a_1^j, a_2^j, \ldots$ of the positive integers such that the restrictions $\eta_{R_{a_n^j}}|_{K_j}$ converge weakly, as $n\to\infty$, to some measure on $K_j$.  In fact, we can choose these such that $\{a_n^j\}_{n=1}^{\infty}$ is a subsequence of $\{a_n^{j'}\}_{n=1}^{\infty}$, for $j'<j$.  

Let $a_n = a_n^n$ be the diagonal sequence.  Now let $f$ be a compactly supported function on $T_1X'$.   Its support is contained in some $K_{j'}$.  It follows that $\lim_{n\to\infty}  \int_{K_j} f d\eta_{R_{a_n^j}}$ exists for all $j\ge j'$, and the value is independent of $j$, because the subsequences are nested.  This gives us a linear functional on the set of compactly supported continuous function $C_c(T_1X')$, which, by the Riesz-Markov theorem, corresponds to a measure $\eta$.  By the definition of weak convergence, the measures $\eta_{R_{a_n}}$ converge to $\eta$.  

To show that $\eta$ is a probability measure, we first note that 
$$\eta(T_1X') = \lim_{n\to\infty} \eta (K_n) = \lim_{n\to\infty} \lim_{k\to\infty}\eta_{R_{a_k}}(K_n) \le 1.$$
For the opposite inequality, we need to rule out ``escape of mass'' towards the points of $\Sigma$.  Fix $\epsilon>0$.  For each $p\in \Sigma$, let $S_{\epsilon}^p$ be a polygon as in Lemma \ref{lemma:poly_inter}, which contains $p$ in its interior.  Let $S_{\epsilon} := \bigcup_{p\in \Sigma} S_{\epsilon}^p$, and let $\widehat S_{\epsilon}$ be the subset of $T_1X'$ consisting of vectors lying over points of $S_{\epsilon}$.  

From the definition of the measures $\mu_{R_n}$ and Lemma~\ref{lemma:poly_inter}, we have 
  \begin{align*}
    \eta_{R_n}(\widehat S_{\epsilon}) & = \mu_{R_n}(S_{\epsilon}) = \frac{1}{N(X,R_n)} \sum_{|s| \le R_n} \mu_s( S_{\epsilon}) = \frac{1}{N(X,R_n)} \sum_{|s| \le R_n} \sum_{p\in \Sigma} \frac{|s\cap S_{\epsilon}^p|}{|s|} \\
    &\le \frac{|\Sigma|}{N(X,R_n)} \left( \sum_{|s| \le R_n} \frac{4k\sqrt{2}\epsilon(\epsilon+1)^2}{\ell(X)^2} + C\right),
  \end{align*}
where $C$ is a constant that absorbs the saddle connections $s$ that are parallel to sides of some $S_{\epsilon}^p$.  We can choose all the $S_{\epsilon}$ to be dilations of each other, so there are only finitely many of these bad saddle connections total, hence $C$ can be chosen independently of $\epsilon$ and $n$.  

Now we can find a $j$ such that $K_j$ contains the complement of $\widehat S_{\epsilon}$.  So by the above,
$$\eta(T_1X') \ge \eta(K_j) = \lim_{n\to\infty}\eta_{R_{a_n}}(K_j) \ge 1-\lim_{n\to\infty}\eta_{R_{a_n}}(\widehat S_{\epsilon})  \ge 1- \frac{4|\Sigma|k\sqrt{2}\epsilon (\epsilon+1)^2}{\ell(X)^2}.$$
Since the last term can be made arbitrarily close to $1$ by choosing $\epsilon$ small, we conclude that $\eta(T_1X') \ge 1$, and we are finished showing that $\eta$ is a probability measure.  

\end{proof}



\subsection{Completing the proof of Theorem \ref{thm:equidist}}

\begin{proof}[Proof of Theorem \ref{thm:equidist}]
  Assume the contrary.  Then, by Banach-Alaoglu, we can find a sequence $R_n\to\infty$ such that $\mu_{R_n} \to \mu$ (weakly), where $\mu$ is a probability measure that is not area measure.  


By Lemma \ref{lemma:no-escape} above, we can find a sequence $\{a_n\}$ of positive integers, such that $\eta_{R_{a_n}}$ converges weakly to a probability measure $\eta$ on $T_1X'$.  Now $T_1X'$ admits maps
  \begin{eqnarray*}
    p: T_1X' \to X', \\
    A: T_1X' \to S^1,
  \end{eqnarray*}
where the first is the obvious projection, and the second takes a tangent vector to its angle.  Note that the push-forward $p_*\eta_R = \mu_R$, and $A_*\eta_R = \nu_R$, the measure from Theorem \ref{thm:angle-lebesgue}. 

By our initial assumption $p_*\eta = \mu$ is not area measure, while by Theorem \ref{thm:angle-lebesgue}, the pushforward $A_*\eta$ is in the Lebesgue measure class.  

Now we can disintegrate the measure $\eta$ along the map $A$, which gives for each $\theta \in S^1$ a measure $\eta_{\theta}$ on $T_1X'$, supported on the fiber $A^{-1}(\theta)$, such that for any continuous, compactly supported $f:T_1X' \to \mathbb{R}$, 

$$\int_{T_1X'} f d\eta = \int_{S^1} \left( \int _{T_1X'} f d\eta_{\theta} \right) dA_*\eta. $$

Now $T_1X'$ supports a geodesic flow.  However, this flow is not defined for all time, since some orbits land on singular points $\Sigma$ after finite time (and the flow cannot be continued in a well-defined manner afterwards).  We will say that a measure $\pi$ on $T_1X'$ is \emph{locally invariant} if for any Borel set $U\subset T_1X'$ such that all the tangent vectors lying over points of $X'$ that are greater than distance $d$ from $\Sigma$, the time $d$, as well as time $-d$, geodesic flow of $U$ (which is well-defined) has the same $\pi$ measure as $U$.  We define local invariance of a measure $\sigma$ on $X'$ under the $\theta$-directional flow (for some fixed $\theta$) in an analogous manner.  

Now the $\eta_s$ are locally invariant, since they come from individual saddle connections.  It follows that the $\eta_R$, and hence $\eta$, are also locally invariant.  We then see that for $A_*\eta$ a.e. $\theta$, the fiber measure $\eta_{\theta}$ is locally invariant, and for such $\theta$ the projection $p_*\eta_{\theta}$ is locally invariant with respect to the $\theta$-directional flow on $X'$.  

By the main result of \cite{kms1986}, for Lebesgue almost every $\theta \in S^1$, the directional flow in direction $\theta$ is uniquely ergodic.  Hence for each such $\theta$, if we could show that $p_*\eta_{\theta}$ is actually invariant under directional flow, rather than merely locally invariant, we could conclude that it has to equal to the area measure $\lambda$ on $X'$.  Since we know $p_*\eta_{\theta}$ is locally invariant, the only possible obstruction to full invariance is if $p_*\eta_{\theta}$ were to give positive mass to the set of points of $X'$ that hit $\Sigma$ under $\theta$-directional flow.  If $\theta$ is not a saddle connection direction, then this implies that some infinite ray, infinite in the backwards direction, and with forward endpoint on $\Sigma$, has positive $p_*\eta_{\theta}$ mass.  But by local invariance (applied in the backwards direction), this ray would have infinite $p_*\eta_{\theta}$ measure, while $p_*\eta_{\theta}$ is a probability measure.   Since the set of saddle connection directions has zero Lebesgue measure, we conclude that for Lebesgue almost every $\theta$, the measure $p_*\eta_{\theta}$ equals $\lambda$.  

 Since $A_*\eta$ is absolutely continuous with respect to Lebesgue measure, $p_*\eta_{\theta}=\lambda$    holds for $A_*\eta$ almost every $\theta$.   Now we will show that this implies that $p_*\eta$ must be area measure, which will be a contradiction.  Let $g:X\to\mathbb{R}$ be any continuous compactly supported function. Then, using the definition of push-forward measure, and the disintegration along $A$, we get 
\begin{align*}
  \int_{X'} g\ dp_*\eta = \int_{T_1X'} g\circ p\ d\eta = \int_{S^1} \left(\int_{T_1X'} g\circ p \ d\eta_{\theta} \right) dA^*\eta &= \int_{S^1} \left( \int_{X'} g \ d p_*\eta_{\theta} \right) dA_* \eta\\
  &= \int_{S^1} \left( \int_X g \ d\lambda \right) dA_*\eta =\int_X g \ d\lambda,  
\end{align*}
so $p_*\eta$ is area measure, contradiction.
\end{proof}

We now describe how the proof of Theorem \ref{thm:equidist} can be modified to work with any subset of saddle connections that is growing quadratically.  

\begin{proof}[Proof sketch of Theorem \ref{thm:equidist-subset}]
  We suppose, to the contrary, that there is some sequence $R_n\to\infty$ such that $\mu_{R_n,S}\to \mu$, where $\mu$ is a probability measure that is not area measure.  As in the proof above, by moving to the unit tangent bundle and passing to a further subsequence, we get a measure $\eta$ on $T_1X'$.  By Theorem \ref{thm:upper-bound} (Upper bound), the pushforward $A_*\eta$ to $S_1$ is absolutely continuous with respect to Lebesgue measure on $S_1$; this is the place where we use the fact that the collection $S$ grows at least quadratically.  Note that, in contrast to the case where $S$ is the whole set of saddle connections, Lebesgue measure need not be absolutely continuous with respect to $A_*\eta$.  For instance, we could take $S$ to be the set of saddle connections whose angle is in $[0,\pi]$, in which case $A_*{\eta}$ would be supported on $[0,\pi]$.  

The rest of the proof is the same - we only ever use that the angle measure $A_*\eta$ is absolutely continuous with respect to Lebesgue, which allows us to apply \cite{kms1986}.  
\end{proof}

\section{Proof of Proposition \ref{prop:circle-decay} via system of integral inequalities}
\label{sec:system-inequalities}

\subsection{Outline of proof}
The proof is typical of the system of integral inequalities approach, and we follow Eskin-Masur (\cite{em2001}, Proof of Proposition 7.2) closely, with the new element being that we work over a fixed interval of angles.  In the process, we will end up proving the sharper and more general Proposition \ref{prop:circle-all-complexes}.  

Here is a sketch of the proof of Proposition \ref{prop:circle-decay}, and an outline of the rest of the paper.  The first several steps do not involve the interval $I$ of angles.  

\begin{itemize} 
\item Suppose (unrealistically) that the shortest saddle connection on each $g_Tr_{\theta}X$  is just the image of the shortest saddle connection $s$ on $X$ under $g_Tr_{\theta}$.  A straightforward $SL_2(\R)$ calculation gives that in this case we would actually get exponential decay in $T$ of the integral of $1/\ell$ over the whole circle of radius $T$.  This calculation is done in Section \ref{subsec:single-vector}.  

\item We will then use a pointwise argument to take care of the case in which the shortest vector on $g_Tr_{\theta}X$ comes from some other $s'$ on $X$.  The idea is to combine $s,s'$ into a complex whose boundary consists of short saddle connections.  The necessary facts about combining complexes are proved in Section \ref{subsec:complexes}.  

\item Our goal is then prove a generalization of the desired theorem with $1/\ell$ replaced by $\alpha_k$, which is defined to be the smallest $y$ such that all complexes of complexity $k$ (and some upper bound on area) have a saddle connection of length at least $y$.   This generalization is stated at the end of Section \ref{subsec:complexes}. 

\item In Section \ref{subsec:aver-over-bound} we prove a bound for the integral of $\alpha_k$ over a circle of fixed radius $\tau$ in terms of the values of $\alpha_j$, for $j\ge k$, at the center of the circle.  This uses the pointwise argument that involves combining complexes. Unfortunately, there are large terms in the inequality that involve $\tau$. 

\item To get around this dependence on the radius $\tau$, in Section \ref{subsec:aver-over-large}, we move out in many steps of size $\tau$ to get to a large arc of radius $T$, so we can then think of $\tau$ as some constant.  This involves some hyperbolic geometry estimates.  This is the first part where we deal with arcs instead of the whole circle, and certain complications arise.  In particular, we have to replace the interval $I$ with a new interval $J$ of the same length when we move out in steps.  

\item Finally, we put everything together in Section \ref{subsec:proof}. This will involve downwards induction on $k$, so that we can deal with the higher complexity error terms that pop up.  This step also has some new complications because we are looking at arcs rather than circles.  

\end{itemize}

\subsection{Decay for a single vector}
\label{subsec:single-vector}

In this section we fix a saddle connection $s$ on $X$ and consider the length of the corresponding saddle connection on $g_tr_{\theta}X$.   This is really just a question about $SL_2(\mathbb{R})$.

\begin{prop}
\label{prop:ave-vector}
  Fix $0\le \delta <1$, and let $v$ be a vector in $\mathbb{R}^2$.  Then 
$$\int_0^{2\pi} \frac{1}{\|g_tr_{\theta}v\|^{1+\delta}}d\theta  \le c(\delta) \frac{e^{-t(1-\delta)}}{\|v\|^{1+\delta}},$$
for all $t\ge 0$, where $c(\delta)$ is a constant depending only on $\delta$.  
\end{prop}

\begin{proof}
  By rotating and scaling, we can assume that $v=(1,0)$.  Also, by symmetry, it suffices to consider $[-\pi/2,\pi/2]$ instead of $[0,2\pi]$ as the domain of integration.  Then 
$$\int_{-\pi/2}^{\pi/2} \frac{1}{\|g_tr_{\theta}v\|^{1+\delta}} = \int_{-\pi/2}^{\pi/2} (e^{-2t}\cos^2 \theta + e^{2t} \sin^2 \theta)^{-(1+\delta)/2} d\theta. $$

We now divide $[-\pi/2,\pi/2]$ into two pieces:
$$S_1 = \{\theta: e^{-2t} \cos^2 \theta > e^{2t}\sin^2 \theta \}$$
$$S_2 = \{\theta: e^{-2t} \cos^2 \theta < e^{2t}\sin^2 \theta \}.$$
Note that $\theta\in S_1$ iff $\theta$ is very close to $0$, and for such angles $\sin \theta \approx \theta$ and $\cos \theta \approx 1$.  It follows that there are constants $c_1,c_2$ such that 
$$[-c_1e^{-2t}, c_1 e^{-2t}] \subset S_1 \subset [-c_2e^{-2t},c_2e^{-2t}].$$
Now
\begin{eqnarray*}
  \int_{S_1} (e^{-2t}\cos^2 \theta + e^{2t} \sin^2 \theta)^{-(1+\delta)/2} d\theta &\le & \int_{S_1} (e^{-2t}\cos^2 \theta)^{-(1+\delta)/2} d\theta \\
&= &O\left(e^{(1+\delta)t} |S_1|\right) \\
&= &O\left(e^{(1+\delta)t} (2c_2e^{-2t})\right) \\
  &= &O(e^{-t(1-\delta)}).
\end{eqnarray*}
It remains to prove a similar bound for the other part of the domain of integration.
\begin{eqnarray*}
  \int_{S_2} (e^{-2t}\cos^2 \theta + e^{2t} \sin^2 \theta)^{-(1+\delta)/2} d\theta &\le & \int_{S_2} (e^{2t} \sin^2\theta)^{-(1+\delta)/2}d\theta \\
  &=& e^{-t(1+\delta)}\int_{S_2} |\sin\theta|^{-(1+\delta)}d\theta \\
  &\le & c \cdot e^{-t(1+\delta)} \int_{S_2} |\theta|^{-(1+\delta)} d\theta \\
  &\le & 2c \cdot e^{-t(1+\delta)} \int_{c_1e^{-2t}}^{\pi/2} |\theta|^{-(1+\delta)} d\theta \\
  &\le & c' \cdot e^{-t(1+\delta)} \left[-|\theta|^{-\delta}\right]|_{c_1e^{-2t}}^{\pi/2}\\
  &=& c'' \cdot e^{-t(1-\delta)}. 
\end{eqnarray*}
Putting together the two bounds yields the desired result.  
\end{proof}

\subsection{Complexes of saddle connections}
\label{subsec:complexes}

In this section we study certain collections of saddle connections called complexes.  There is no bound over the whole stratum on the number of saddle connections on $X$ shorter than some specified $\epsilon$, since one can find ``small'' subsurfaces (eg a small cylinder) which contain lots of short saddle connections.  To get around this, we build complexes, which have a notion of complexity that cannot increase indefinitely.  

\begin{defn}
  A \emph{complex} $K$ in $X$ is a closed subset of $X$ whose boundary $\partial K$ consists of a union of disjoint saddle connections (when we say that two saddle connections are disjoint, we mean that the interiors are disjoint), such that if $\partial K$ contains three saddle connections that bound a triangle, then the interior of that triangle is in $K$.   We will denote by $|\partial K|$ the length of the longest saddle connection in $\partial K$.  
\end{defn}

\begin{defn}
  Given a complex $K$, the \emph{complexity} of $K$ is the number of saddle connections needed to triangulate $K$ (a triangulation of a complex $K$ is a collection $S$ of saddle connections together with a collection of triangles $T$, with disjoint interiors, each of whose boundaries consist of elements of $S$, such that $K=S\cup T$).  
\end{defn}

Note, by an Euler characteristic argument, the number of saddle connections in any triangulation of $K$ is independent of the triangulation. 

Informally, Proposition \ref{prop:comb-complex} below says that if we have a complex whose boundary consists of short saddle connections, and we find a short saddle connection that is either disjoint from the complex or crosses the boundary, then we can extend the complex to a complex of higher complexity whose boundary saddle connections are still short.  See Figure \ref{fig:combining_complex}.  

 \begin{prop}
  Suppose $K$ is complex of complexity $k$ with non-empty boundary, and let $\sigma$ be a saddle connection which is either disjoint from $K$ or crosses $\partial K$.  Then there is a complex $K'\supset K$ of complexity $i>k$ such that 
  \begin{enumerate}[(i)]
  \item $| \partial K' | \le |\partial K | + \max(|\partial K|, |\sigma|), $
  \item $\operatorname{area}(K') \le \operatorname{area}(K) + |\partial K| \max(|\partial K|, |\sigma|).$
  \end{enumerate}

\label{prop:comb-complex}
\end{prop}

\begin{proof}[Proof sketch (see \cite{em2001}, Proposition 6.2, for full details):]

  We will give the proof in some (but not all) of the possible cases, to give a reader a flavor of the argument.  
  If $\sigma$ is disjoint from $K$, then we just take $K'=K\cup \sigma$ which has complexity $k+1$ and the length of the boundary clearly satisfies the desired inequality (i).  The area does not increase, so (ii) is also satisfied.  

If $\sigma$ crosses $\partial K$, there are various cases to work out.  

Consider first the case where $\sigma$ has one endpoint in $K$, and one endpoint, $p$, outside of $K$.  Moving from $p$ along $\sigma$, let $s$ be the first saddle connection in $K$ that $\sigma$ hits.  Let $\sigma'$ be the segment of $\sigma$ that goes from $p$ to $s$.  By keeping the endpoint $p$ of $\sigma'$ fixed, and moving the other endpoint of $\sigma$ a small amount along $s$, we get a family of new straight segments that don't hit singular points (except at $p$).  Eventually, as we move the endpoint along $s$, the segment will hit some singular point $q$ (and $q$ is the first singular point hit, if we start at $p$ and move along the segment).  If $q$ is not in $K$, then the segment between $p$ and $q$ is disjoint from $K$ (since as we move $\sigma'$, the first time it hits $K$ in the interior of the segment must be at a singular point of $K$), so we can add the segment, which is in fact a saddle connection, to the complex.  If $q$ is in $K$, then it must be on the boundary of $K$, and again we can the saddle connection connecting $p$ to $q$ to $K$ to get a new complex.  In either case, we may have to add more triangles to the complex if the three boundary saddle connections are already in the complex.  In both cases, condition (i) is satisfied, since by the triangle inequality, the new saddle connection added has length at most $|\sigma|+|s|$, and from this we see that condition (ii) is also satisfied.  


The other case is similar.  
\end{proof}

\begin{figure}[h]
\begin{center}
  \includegraphics[scale=0.8]{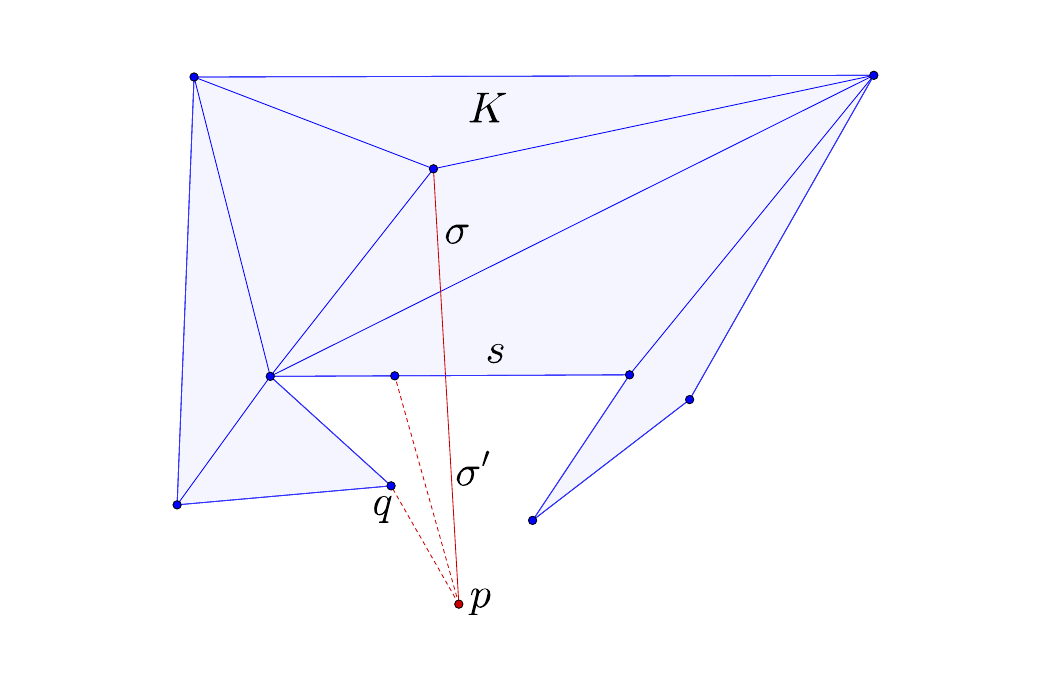}
  \caption{Adding a saddle connection to a complex, in proof of Proposition \ref{prop:comb-complex}.} \label{fig:combining_complex}
\end{center}
\end{figure}

While we are primarily interested in studying the function $1/\ell(X)$, where $\ell(X)$ is the length of the shortest saddle connection on $X$, we will be forced to also consider complexes in which all the boundary saddle connections are short. 

Fix $\delta>0$.  Let $\beta=1/2^{M+1}$, where $M$ is the complexity of $X$. We define a sequence of functions
$$\alpha_i(X)= \max_K \frac{1}{|\partial K|^{1+\delta}}, $$
where the $\max$ is taken over 
\begin{eqnarray}
  \label{eq:area-assumption}
  \{ K : K \text{ complex in } X \text{ of complexity } i, \text{ } \area(K)<2^i\beta\}. 
\end{eqnarray}
For some $i,X$, if there are no complexes satisfying (\ref{eq:area-assumption}), then we set $\alpha_i(X)=0$.   We need the area restriction to keep the complexes from getting too big - in particular, we need to avoid considering a complex that is equal to the whole surface.   Note that $\alpha_1(X)= 1/\ell(X)^{1+\delta}$, since a complex of complexity $1$ must be a single saddle connection, and this has zero area.    

The following theorem is the main technical result of this article.  The proof will be completed in Section \ref{subsec:proof}.  

\begin{prop}
  Fix a stratum $\mathcal{H}$, and $0<\delta < 1/2$.  We can find a constant $b$ such that for any interval $I\subset S^1$, there exists a constant $c_{I}$ such that for all $X \in \mathcal{H}$, 
$$\int_I \alpha_k(g_Tr_{\theta}X)d\theta < c_I\cdot e^{-(1-2\delta)T}\sum_{j\ge k}\alpha_j(X) + b \cdot |I|,$$
for all $T\ge 0$. 
\label{prop:circle-all-complexes}
\end{prop}

\begin{proof}[Proof of Proposition \ref{prop:circle-decay} assuming Proposition
\ref{prop:circle-all-complexes}:]
This is just the case $k=1$.  
\end{proof}

\subsection{Averaging over a circle of bounded size}
\label{subsec:aver-over-bound}

Given a function $f$ on $\mathcal{H}$ and a point $X\in \mathbb{H}$, we let
$$\operatorname{Ave}_t(f)(X) := \frac{1}{2\pi} \int_0^{2\pi} f(g_tr_{\theta}X)d\theta. $$

\begin{prop}
  Fix $\mathcal{H}$ and $0<\delta<1$.  There exists $C>0$, such that for any $t>0$, there exist constants $b_t,w_t$ such that for any $k$ and $X\in \mathcal{H}$,
$$\operatorname{Ave}_t(\alpha_k)(X) \le C e^{-t(1-\delta)}\alpha_k(X)  + w_t \sum_{j>k} \alpha_j(X) + b_t.$$
\label{prop:bounded-size}
\end{prop}

\begin{proof}
  Let $K$ be a complex on $X$ of complexity $k$ realizing the definition of $\alpha_k(X)$, and let $K'=K'(\theta)$ be a complex on $g_tr_{\theta}X$  realizing $\alpha_k(g_tr_{\theta}X)$.  

Here is the idea of the proof.  The first term on the right hand side of the bound comes from the case when $K'=g_tr_{\theta}(K)$; in this case we get a bound from Proposition  \ref{prop:ave-vector}.  The second term comes from the case where $K'$ is some other complex; in this case we get a bound by combining $(g_tr_{\theta})^{-1}K'$ with $K$ to get a higher complexity complex. For this we assume that $\partial K$ consists of short saddle connections.  The third term handles the case when none of the saddle connections in $K$ are short.  
  
Let $E\subset [0,2\pi)$ be the set of $\theta$ for which $K'=g_tr_{\theta}K$, and let $F$ be the complement.  Then
\begin{eqnarray*}
  \operatorname{Ave}_t(\alpha_k)(X) &=& \frac{1}{2\pi}\left( \int_E \alpha_k(g_tr_{\theta}X)d\theta + \int_F \alpha_k(g_tr_{\theta}X)d\theta \right). 
\end{eqnarray*}

\begin{enumerate}

\item To bound the integral over $E$ we apply Proposition \ref{prop:ave-vector}:
\begin{eqnarray}
  \int_E \alpha_k(g_tr_{\theta}X)d\theta &=& \int_E \min_{s\in \partial K} \frac{1}{\|g_tr_{\theta}s\|^{1+\delta}} d\theta \\
&\le & \min_{s\in \partial K} \int_0^{2\pi} \frac{1}{\|g_tr_{\theta} s\|^{1+\delta}} d\theta \\
  &\le & \min_{s\in \partial K} c(\delta) \frac{e^{-t(1-\delta)}}{\|s\|^{1+\delta}}= c(\delta) e^{-t(1-\delta)}\alpha_k(X). \label{eq:same-vec-bound}
\end{eqnarray}

\item Now we bound the integral over $F$ in a pointwise fashion.  First assume that
\begin{eqnarray*}
   |\partial K| \ge e^{-2t}\sqrt{\beta}.
\end{eqnarray*}
Then we get the bound
\begin{eqnarray}
\int_F \alpha_k(g_tr_{\theta}X)d\theta &\le& \int_0^{2\pi} \alpha_k(g_tr_{\theta}X)d\theta \le \int_0^{2\pi} (e^t)^{1+\delta} \alpha_k(X)d\theta \\
  &\le& 2\pi (e^{t})^{1+\delta} (e^{-2t}\sqrt{\beta})^{-(1+\delta)}=:b_t. 
\label{eq:all-short}
\end{eqnarray}

\item Now assume that
\begin{eqnarray}
   |\partial K| \le e^{-2t}\sqrt{\beta}.
 \label{eq:assume-small}
\end{eqnarray} 
Note that because of our assumption (\ref{eq:area-assumption}) on the area of the complexes that we allow in the definition of $\alpha_k$, we know that neither $K,K'$ is all of the surface, and hence each must have non-empty boundary.  It follows that some saddle connection $\tilde s$ in $\partial (g_tr_{\theta})^{-1}K'$ must either be disjoint from $K$, or cross $\partial K$.  By Proposition \ref{prop:comb-complex}, from $\tilde s$ and $K$ we can form a new complex $\tilde K$ on $X$ with complexity $i>k$ such that 
\begin{eqnarray*} 
|\partial \tilde K | &\le& |\partial K| + \max(|\partial K|,|\tilde s |) \le |\partial K| +\max(|\partial K|, |\partial (g_tr_{\theta})^{-1}K'|) \\
&\le &  2|\partial (g_tr_{\theta})^{-1}K'| \le  2 e^t |\partial K'|,
\end{eqnarray*}
and
\begin{eqnarray}
  \area(\tilde K)  \le \area(K) + |\partial K| \max(|\partial K|, |\tilde s|).
  \label{eq:area}
\end{eqnarray}
Now $K'$ realizes the maximum in the definition of $\alpha_k(g_tr_{\theta}X)$, so in particular $$|\partial K'| \le |\partial (g_tr_{\theta}K)| \le e^t |\partial K| \le e^t(e^{-2t} \sqrt{\beta}) =e^{-t} \sqrt{\beta}.$$
Thus $|\tilde s| \le e^t|\partial K'| \le \sqrt{\beta}$. 
Starting with the inequality (\ref{eq:area}), and using bound on area from (\ref{eq:area-assumption}) together with our assumption (\ref{eq:assume-small}) on $|\partial K|$, we get 
\begin{eqnarray*}
  \area(\tilde K) &\le & 2^k\beta + (e^{-2t}\sqrt{\beta})\max(e^{-2t}\sqrt{\beta}, \sqrt{\beta}) \\
  &  \le & 2^k\beta + \beta <  2^{k+1}\beta  \le 2^i\beta. 
\end{eqnarray*}

So the complex $\tilde K$ is one of those over which the maximum in the definition of $\alpha_i(X)$ is taken.  Also, since $\beta=1/2^{M+1}$, we have $\area(\tilde K) <1$, and hence $\partial \tilde K$ must be non-empty.  It follows that  
$$\alpha_i(X) \ge \frac{1}{|\partial \tilde K|^{1+\delta}} \ge \frac{1}{(2e^t)^{1+\delta}} 
\frac{1}{|\partial K'|^{1+\delta}} = \frac{1}{(2e^t)^{1+\delta}} \alpha_k(g_tr_{\theta}(X)). $$
So 
$$\alpha_k(g_tr_{\theta}(X)) \le (2e^t)^{1+\delta} \alpha_i(X).$$

Using this pointwise bound we get (under the assumption (\ref{eq:assume-small})) that

\begin{eqnarray}
  \int_F \alpha_k(g_tr_{\theta}X)d\theta \le \int_F (2e^t)^{1+\delta}\sum_{i>k}\alpha_i(X) \le 2\pi (2e^t)^{1+\delta} \sum_{i>k} \alpha_i(X). 
\label{eq:new-complex-bound}
\end{eqnarray}

\end{enumerate}

Now we put the three parts together: by  (\ref{eq:same-vec-bound}), (\ref{eq:all-short}), and (\ref{eq:new-complex-bound}), we have
\begin{eqnarray*}
  \operatorname{Ave}_t(\alpha_k)(X) \le \frac{c(\delta)}{2\pi} e^{-t(1-\delta)}\alpha_k(X) + (2e^t)^{1+\delta} \sum_{i>k} \alpha_i(X) + \frac{b_t}{2\pi}.
\end{eqnarray*}

\end{proof}

\subsection{Averaging over larger arcs}
\label{subsec:aver-over-large}
To prove Proposition \ref{prop:circle-all-complexes}, we need to compare the average of the function $\alpha_i$ over an arc from a large circle to the value of $\alpha_i$ at the center of the circle.  Proposition \ref{prop:bounded-size} gives a comparison, but the $w_t$ term means as we make $t$ large we lose control over the size of the average.  To get around this, we move in steps, repeatedly applying Proposition \ref{prop:bounded-size} over circles of some fixed size $\tau$.  To go from an arc of a circle of radius $t$ to one of radius $t+\tau$, we need Lemma \ref{lemma:shadowing} (Shadowing) below.  This is one of the steps in the proof that would be substantially easier if we were working just with the whole circle $I=[0,2\pi]$, rather than with any interval of angles.  

We first need a preliminary hyperbolic geometry lemma.  Let us choose $\kappa >0$ such that for any $X,Y\in \mathcal{H}$ that are in the same $SL_2(\mathbb{R})$ orbit, and satisfy $d(X,Y)<\kappa$, we have $\frac{1}{2} \alpha_i(X) \le \alpha_i(Y) \le 2 \alpha_i(X)$.  Here $d(X,Y)$ is defined as the Teichm\"uller distance between the projections of $X,Y$ to the moduli space.  Such a choice of $\kappa$ is possible because the Teichm\"uller distance of the projections controls the ratio of lengths of corresponding saddle connections on the surfaces.  

Now let $U=U(t,\tau,\kappa):=\{\phi : d(g_{\tau}r_{\phi}g_tX, X) \ge t+ \tau -\kappa\}$.  In Figure \ref{fig:shadowing}, $U$ is the set of angles $\phi$ corresponding to the dashed segment.  

\begin{figure}[h]
\begin{center}
  \includegraphics[scale=0.61]{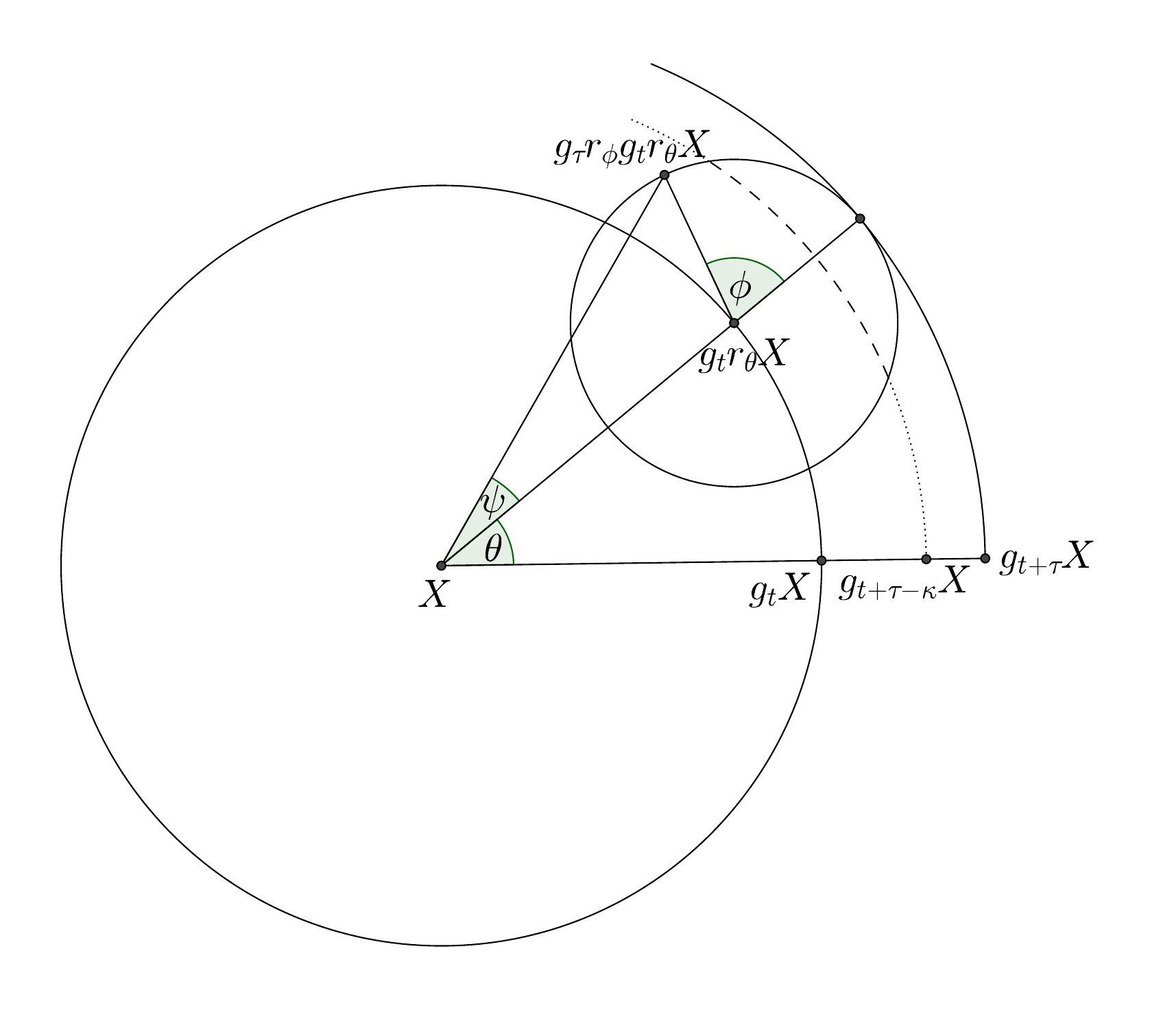}
  \caption{Comparing averages for Lemma \ref{lemma:shadowing} (Shadowing)}    \label{fig:shadowing}
\end{center}
\end{figure}

The following Lemma is closely related to \cite{em2001} Lemma 7.6 and results in \cite{athreya2006} Section \S4.

\begin{lemma}
  Fix $\kappa>0$.  There exists $c'>0$ such that for all $t,\tau$ we have $|U(t,\tau,\kappa)| \ge c'$.  
\label{lemma:angles-outside}
\end{lemma}

The lemma says that, in the hyperbolic plane, if we consider a circle of radius $t$ centered at some point $p$, then at least a definite, positive proportion of any circle of radius $\tau$ centered at a point on the first circle will lie outside the disk of radius $t+\tau -\kappa$ centered at $p$.  The proportion depends on $\kappa$, but not on $t$ or $\tau$.  

\begin{proof}
  Consider the triangle with vertices $X$, $Y:=g_tX$, and $Z:=g_{\tau}r_{\phi}g_tX$.  Note that $|XY|=t$, $|YZ|=\tau$, and $\angle XYZ=\pi - \phi$.  Applying the Hyperbolic Law of Cosines to the triangle $XYZ$ gives 
$$\cosh (XZ) = \cosh(t) \cosh (\tau) + \sinh (t) \sinh(\tau) \cos \phi.$$
Now for large $x$, we have that $\cosh(x) \approx \frac{e^x}{2} \approx \sinh(x)$.  Using this approximation we get that 
$$\frac{e^{|XZ|}}{2} \approx \left(\frac{e^t}{2}\frac{e^{\tau}}{2}\right) (1+\cos \phi),$$
so
$$|XZ| \approx t+ \tau  + \log \frac{1+\cos \phi}{2}.$$
Now $\phi \in U$ iff $|XZ|\ge t+\tau -\kappa$.  If the above approximation is accurate enough, we get that $|XZ|\ge t+\tau-\kappa$, when $|\phi|\le c'$, for some $c'$ independent of $t,\tau$, as desired.  To justify the approximation, we can absorb the error in a multiplicative term that is close to $1$ when $t,\tau$ are large, and this just makes $c'$ smaller by a multiplicative factor. To handle bounded $t,\tau$, we make the constant $c'$ smaller if necessary.  
\end{proof}

\begin{lemma}[Shadowing] 
There exists a constant $c_2>0$ such that for any $\tau\ge 0$ and $I\subset S^1$ an interval, there exists $t_0(\tau, |I|) \ge 0$ such that for any $X\in \mathcal{H}$, and $t>t_0$, we have
$$\int_I \alpha_i(g_{t+\tau}r_{\theta}X) \le c_2 \int_J \operatorname{Ave}_{\tau}(\alpha_i)(g_tr_{\theta}X)d\theta, $$
where $J\subset S^1$ is an interval (that could depend on all the other parameters) with $|J|=|I|$. 
\label{lemma:shadowing}
\end{lemma}

The new interval $J$ is needed to take care of possible edge effects near the boundary of $I$. The reason that we need to take $t$ sufficiently large is also related to edge effects.  So the statement and proof are more complicated than for $I=[0,2\pi]$, since this interval has no boundary. 

\begin{proof}

We want to parametrize the point $g_{\tau}r_{\phi}g_tr_{\theta}X$ by the angle $\psi$ as indicated in the diagram, i.e. we want $g_{\tau}r_{\phi}g_tr_{\theta}X = g_sr_{\theta+\psi}X$, where $s,\psi$ are functions of $t,\tau,\phi$.  Let $\Psi:S^1 \to S^1$ be the map taking $\phi$ to $\psi$.  On small intervals, this map is a diffeomorphism to its image.   Now changing the variable from $\phi$ to $\psi$, and using the defining property of $\kappa$, we get 
\begin{eqnarray*}
  \operatorname{Ave}_{\tau}(\alpha_i)(g_tr_{\theta}X) &\ge & \frac{1}{2\pi} \int_{U} \alpha_i (g_{\tau}r_{\phi}g_tr_{\theta}X)d\phi   \\
  & = & \frac{1}{2\pi} \int_{\Psi(U)} \alpha_i(g_{\tau}r_{\phi(\psi)} g_tr_{\theta}X)\left|\frac{d\phi}{d\psi}\right| d\psi \\
  &\ge & \frac{1}{2\pi} \int_{\Psi(U)}\frac{1}{2} \alpha_i(g_{t+\tau}r_{\theta+\psi}X )\left|\frac{d\phi}{d\psi}\right| d\psi.
\end{eqnarray*}

Now, to estimate the right hand term in the statement of the lemma, we let $A=  \left( \begin{matrix} 1 & 1 \\ 0 & 1 \end{matrix} \right)$ and perform the linear change of variables
$$\left(  \begin{matrix} \theta \\ \psi \end{matrix} \right) \mapsto \left( \begin{matrix} \xi \\ \psi \end{matrix} \right) = A \left( \begin{matrix} \theta \\ \psi \end{matrix} \right) = \left( \begin{matrix} \theta +\psi  \\ \psi \end{matrix} \right),  $$
where the Jacobian of $A$ is $1$.  Let $2I$ be the interval with the same center as $I$ and twice the length. Integrating the inequality above over $\theta$ and performing this change of variables gives 
\begin{eqnarray}
\int_{2I} \operatorname{Ave}_{\tau}(\alpha_i)(g_tr_{\theta}X)d\theta & \ge & \int_{2I} \frac{1}{4\pi} \int_{\Psi(U)} \alpha_i(g_{t+\tau}r_{\theta+\psi}X )\left|\frac{d\phi}{d\psi}\right| d\psi d\theta \\
  &=& \frac{1}{4\pi} \int \int _{A\left(2I \times \Psi(U)\right)} \alpha_i(g_{t+\tau} r_{\xi} X)  \left|\frac{d\phi}{d\psi}(\psi)\right| d\xi d\psi \\
 &=& \frac{1}{4\pi} \int_{\Psi(U)} \int_{\psi + 2I} \alpha_i(g_{t+\tau} r_{\xi} X) \left|\frac{d\phi}{d\psi}(\psi)\right| d\xi d\psi \\
  &=& \frac{1}{4\pi}  \int_{\Psi(U)}\left( \left|\frac{d\phi}{d\psi}(\psi)\right|\int_{\psi + 2I} \alpha_i(g_{t+\tau}r_{\xi}X) d\xi \right) d\psi.  \label{eq:change-var}
\end{eqnarray}
Recall that $U$, and hence $\Psi(U)$, depends on $t,\tau$.  Now $\Psi(U)$ is an interval, and the largest angle $\gamma\in \Psi(U)$ corresponds to a ray through $X$ that is tangent to the smaller circle.  Hence we get a hyperbolic right triangle with hypotenuse of length $t$, a leg of length $\tau$, and the angle opposite this leg being $\gamma$.  Hence, by hyperbolic trigonometry, 
$$\sin \gamma = \frac{\sinh \tau}{\sinh t}.$$
We are considering $\tau$ as fixed for now.  Choose $t_0$ sufficiently large so that for any $t\ge t_0$ we have $|\gamma|< |I|/2$.  For such $t$, for any $\psi \in \Psi(U)$, we get that $|\psi| \le |\gamma|  < |I|/2$, and so $\psi + 2I \supset I$.  Continuing from (\ref{eq:change-var}), we get
\begin{eqnarray}
  \int_{2I} \operatorname{Ave}_{\tau}(\alpha_i)(g_tr_{\theta}X)d\theta &\ge&  \frac{1}{4\pi}  \int_{\Psi(U)}\left( \left|\frac{d\phi}{d\psi}(\psi)\right|\int_{I} \alpha_i(g_{t+\tau}r_{\xi}X) d\xi \right) d\psi \\
&=& \frac{1}{4\pi} \left( \int_{\Psi(U)}\left|\frac{d\phi}{d\psi}(\psi)\right| d\psi \right) \left(\int_{I} \alpha_i(g_{t+\tau}r_{\xi}X) d\xi \right)\\
&=&   \frac{1}{4\pi} \left( \int_U d\phi \right) \int_{I} \alpha_i(g_{t+\tau}r_{\xi}X) d\xi\\
&\ge &  \frac{c'}{4\pi}\int_{I} \alpha_i(g_{t+\tau}r_{\xi}X) d\xi, \label{eq:twice-interval}
\end{eqnarray}
where we have used Lemma \ref{lemma:angles-outside} for the last inequality.  

Now we can write $2I$ as the union of two intervals $J_1,J_2$, with $|J_1|=|J_2|=|I|$.  Then
\begin{eqnarray*}
  \max_{j=1,2} \int_{J_i} \operatorname{Ave}_{\tau}(\alpha_i)(g_tr_{\theta}X) \ge \frac{1}{2} \int_{2I} \operatorname{Ave}_{\tau}(\alpha_i)(g_tr_{\theta}X)
\end{eqnarray*}
Let $J$ be the interval on which the max is achieved.  Combining with (\ref{eq:twice-interval}), we get
\begin{eqnarray*}
  \int_{J} \operatorname{Ave}_{\tau}(\alpha_i)(g_tr_{\theta}X) \ge \frac{1}{2} \int_{2I} \operatorname{Ave}_{\tau}(\alpha_i)(g_tr_{\theta}X) \ge  \frac{c'}{8\pi}\int_{I} \alpha_i(g_{t+\tau}r_{\xi}X) d\xi,
\end{eqnarray*}
which yields the desired result.

\end{proof}

\subsection{Completing the proof}
\label{subsec:proof}

\begin{proof}[Proof of Proposition \ref{prop:circle-all-complexes}]
The strategy is to use Proposition \ref{prop:bounded-size} repeatedly, moving out to a large radius arc, applying Lemma \ref{lemma:shadowing} for comparisons along the way.   This will give an upper bound on the arc integral that has one summand that is an exponentially decaying term times $\alpha_k(X)$, plus a mess of higher complexity terms and constants.  At every stage, we accumulate higher complexity terms $\alpha_j$, $j>k$.  When we have done $n$ stages, there are $n$ higher complexity terms, each of which is a product of a higher complexity term from an arc corresponding to $i$ steps, multiplied by $n-i$ decaying terms coming from Proposition \ref{prop:bounded-size}.  The $i$ step arc part we control by downwards induction on $k$, and the $n-i$ decaying terms each decay a bit faster than a $1$ step arc term would.  The aggregate of the $n$ higher complexity terms is thus bounded by a geometric series times $\sum_{j>k} \alpha_j(X)$, and the series decays around as fast as the $\alpha_k(X)$ term.  We can't get all the way back to the center $X$, because of the lower bound on $t$ needed to apply Lemma \ref{lemma:shadowing} (this is an additional place where the proof is more complicated than for $I=[0,2\pi]$).  So we stop at some arc whose radius depends on $|I|$, and then approximate $\alpha_k$ on this arc by a large factor (depending on $I$) multiplied by $\alpha_k(X)$.   

So we proceed by downwards induction on $k$.  The statement is trivial when $k>M$ (recall that $M$ is the complexity of the whole surface $X$), because for this $k$ we have defined $\alpha_k(X)=0$ for all $X$.  So assume the result holds for all $j>k$, and we will prove it for $k$.  Consider fixing $\tau>0$.  Let $m$ be the smallest positive integer such that $(m-1)\tau > t_0(\tau,|I|)$, where $t_0$ is the function given by Lemma \ref{lemma:shadowing} (Shadowing). Using Lemma \ref{lemma:shadowing}, and then Proposition \ref{prop:bounded-size}, we have that for any $n\ge m$,  
\begin{eqnarray*}
  \int_I \alpha_k(g_{n\tau} r_{\theta} X) d\theta &\le& c_2 \int_J \operatorname{Ave}_{\tau}(\alpha_k)(g_{(n-1)\tau}r_{\theta}X) d\theta \\
  &\le & c_2\int_J \left( C e^{-\tau(1-\delta)} \alpha_k(g_{(n-1)\tau}r_{\theta}X) + w_{\tau} \sum_{j>k} \alpha_j(g_{(n-1)\tau}r_{\theta}X) + b_{\tau} \right) d\theta\\
&=& c_2Ce^{-\tau(1-\delta)} \int_J \alpha_k(g_{(n-1)\tau}r_{\theta}X) d\theta + c_2w_{\tau}\left( \sum_{j>k}\int_J \alpha_j(g_{(n-1)\tau}r_{\theta}X)\right)d\theta + |J| c_2b_{\tau} .
\end{eqnarray*}
  Applying the inductive hypothesis to the terms in the middle sum, enlarging the constants $b_{\tau}, w_{\tau}=w_{\tau,I}$, and using the fact that $|J|=|I|$, we get 
  \begin{eqnarray*}
   \int_I \alpha_k(g_{n\tau} r_{\theta} X) d\theta &\le & c \cdot e^{-\tau(1-\delta)} \int_J \alpha_k(g_{(n-1)\tau} r_{\theta}X) d\theta +  w_{\tau,I} \left( \sum_{j>k} (e^{-(1-2\delta)})^{(n-1)\tau} \alpha_j(X) \right) +  |I| b_{\tau},
\end{eqnarray*}
where the constant $c$ is independent of $\tau$.  This holds for any $n\ge m$, and we can replace $I$ by any interval of the same length (which might entail a different $J$, but $m$ only depends on the length $|I|$).  
Let $J_{n-1}=J$.  Repeatedly applying the inequality for $n~-~1,{n-2},\ldots,m$, with $I$ replaced by $J_{n-1}, J_{n-2},\ldots, J_m$ (the new intervals we get at each stage, which all have length $|I|$), we get 
\begin{align*}
    \int_I \alpha_k(g_{n\tau}r_{\theta}X)d\theta &\le  \left(c e^{-\tau(1-\delta)}\right)^{n-m+1} \int_{J_{m-1}} \alpha_k(g_{(m-1)\tau}r_{\theta}X) d\theta\\
  &+ w_{\tau,I} \sum_{j> k} \left( e^{-\tau(1-2\delta)(n-1)}  + e^{-\tau(1-2\delta)(n-2)}ce^{-\tau(1-\delta)} + e^{-\tau(1-2\delta)(n-3)}(ce^{-\tau(1-\delta)})^2 + \cdots \right) \alpha_j(X) \\
&+ |I| b_{\tau} \left( 1+ ce^{-\tau(1-\delta)}+(ce^{-\tau(1-\delta)})^2 + \cdots\right).
  \end{align*}
The common ratio in the first geometric series is $ce^{-\tau \delta}$, while in the second it is $ce^{-\tau(1-\delta)}$.  By choosing $\tau$ sufficiently large, we can make both of these ratios less than $1/2$ (here we use that $\delta>0$), and then each of the series is bounded by (twice) the first term.  For new constants $w_{\tau},b_{\tau}$, we get 
\begin{align}
    \int_I \alpha_k(g_{n\tau}r_{\theta}X)d\theta &\le  \left(c e^{-\tau(1-\delta)}\right)^{n-m+1} \int_{J_{m-1}} \alpha_k(g_{(m-1)\tau}r_{\theta}X) d\theta  + w_{\tau,I} \cdot \sum_{j>k} \alpha_j(X)  e^{-\tau(1-2\delta)(n-1)}  + |I| b_{\tau}. \label{eq:many-steps}
  \end{align}

Since $\alpha_k$ is defined in terms of lengths of saddle connections, $\alpha_k(g_{(m-1)\tau}r_{\theta}X) \le e^{(m-1)\tau} \alpha_k(X),$ and so
\begin{eqnarray*}
  \int_{J_{m-1}} \alpha_k(g_{(m-1)\tau}r_{\theta}X) d\theta \le (2\pi) e^{(m-1)\tau} \alpha_k(X). 
\end{eqnarray*}

Now take $\tau$ large enough so that $c< e^{\tau\delta}$ and continue from (\ref{eq:many-steps}) to get 
\begin{align*}
 \int_I \alpha_k(g_{n\tau}r_{\theta}X)d\theta &\le  \left(c e^{-\tau(1-\delta)}\right)^{n-m+1} (2\pi) e^{(m-1)\tau} \alpha_k(X) + w_{\tau,I} \cdot \sum_{j>k} \alpha_j(X)  e^{-\tau(1-2\delta)(n-1)}  + |I| b_{\tau} \\
  &\le  \left( e^{-\tau(1-2\delta)}\right)^n \left( e^{-\tau(1-2\delta)}\right)^{-m+1}  (2\pi) e^{(m-1)\tau} \alpha_k(X) + w_{\tau,I} \cdot e^{\tau(1-2\delta)} \sum_{j>k} \alpha_j(X)  e^{-\tau(1-2\delta)n}  + |I| b_{\tau} \\
  &\le  c_{\tau,m} e^{-\tau(1-2\delta)n} \sum_{j\ge k} \alpha_j(X) + |I| b_{\tau}, 
  \end{align*}
where $c_{\tau,I}$ is a constant depending on $\tau,I$ (this constant also absorbs the $w_{\tau,I}$ term in the above).  Taking $T=n\tau$ gives the desired result. 
\end{proof}

\section{Open questions}
\label{sec:open-questions}

\begin{question}
  Do the measures $\nu_R$ on $S^1$ coming from angles of saddle connections converge weakly to Lebesgue measure as $R\to\infty$, for every surface $X$?  
  \label{question:angle-conv}
\end{question}

We could answer this affirmatively using the Eskin-Masur counting strategy if we knew that for every interval $I$, the ``arcs'' $\{g_Tr_{\theta}X: \theta \in I \}$ became equidistributed in the orbit closure $\overline{SL_2(\mathbb{R})}X$ as $T\to\infty$.  The question of equidistribution of arcs is connected to horocycle flow invariant measures on $\mathcal{H}$, since a large arc in a certain sense looks like a long piece of horocycle.  Proving results about the structure of horocycle flow invariant measures has proved to be a difficult challenge.

\begin{question}
  Do the measures $\eta_R$ on $T_1X'$, defined in Section \ref{sec:proof-cor}, converge weakly to the uniform measure on $T_1X'$ (that is, the measure which is locally the product of area measure on $X'$ and the Lebesgue measure on the tangent circle coming from the metric; in some contexts, this is known as Liouville measure)?  
  \label{question:tangent-conv}
\end{question}

Since $\nu_R$ is the push-forward of $\eta_R$ under the map taking a tangent vector to its angle, an affirmative answer to Question \ref{question:tangent-conv} implies an affirmative answer to Question \ref{question:angle-conv}.  On the other hand, if we know for a particular surface $X$ that $\nu_R$ converges weakly to Lebesgue measure on $S^1$, adapting the argument in the proof of Theorem \ref{thm:equidist}, we could conclude that the measures $\eta_R$ for $X$ converge to uniform measure on $T_1X'$.  In particular, since for \emph{almost every} $X$ in a stratum, $\nu_R$ converges to Lebesgue on $S^1$ (see \cite{vorobets2005}, Theorem 1.9), we know that the answer to Question \ref{question:tangent-conv} is yes for almost every $X$.

The following question concerns equidistribution of subsets of saddle connections that could be growing slightly slower than quadratically.  

\begin{question}
  Using the notation of Theorem \ref{thm:equidist-subset}, let $S$ be a subset of saddle connections on $X$ for which 
$$\lim_{R\to\infty} \frac{\log N_S(R)}{\log R}=2.$$
Do the measures $\mu_{R,S}$ on $X$ coming from averaging uniform measure on the elements of $S$ converge weakly to area measure on $X$, as $R\to\infty$?  
\end{question}

For instance, the question encompasses the case where $N_S(R)$ is growing like $\frac{R^2}{\log R}$.  This question is motivated by a phenomenon for hyperbolic surfaces in which collections of closed geodesics whose growth rate is within (for instance) a polynomial factor of the (exponential) growth rate of all the closed geodesics still become equidistributed on the surface.

\bibliography{sources}{}

\providecommand{\bysame}{\leavevmode\hbox to3em{\hrulefill}\thinspace}
\providecommand{\MR}{\relax\ifhmode\unskip\space\fi MR }
\providecommand{\MRhref}[2]{%
  \href{http://www.ams.org/mathscinet-getitem?mr=#1}{#2}
}
\providecommand{\href}[2]{#2}
\begin{thebibliography}{BGKT98}

\bibitem[Ath06]{athreya2006}
Jayadev~S. Athreya, \emph{Quantitative recurrence and large deviations for
  {T}eichmuller geodesic flow}, Geom. Dedicata \textbf{119} (2006), 121--140.
  \MR{2247652}

\bibitem[BGKT98]{bgkt1998}
M.~Boshernitzan, G.~Galperin, T.~Kr{\"u}ger, and S.~Troubetzkoy, \emph{Periodic
  billiard orbits are dense in rational polygons}, Trans. Amer. Math. Soc.
  \textbf{350} (1998), no.~9, 3523--3535. \MR{1458298}

\bibitem[Bow72]{bowen1972}
Rufus Bowen, \emph{The equidistribution of closed geodesics}, Amer. J. Math.
  \textbf{94} (1972), 413--423. \MR{0315742}

\bibitem[{Cha}11]{chaika2011}
J.~{Chaika}, \emph{{Homogeneous approximation for flows on translation
  surfaces}}, arXiv:1110.6167 (2011).

\bibitem[Doz18]{dozier2018}
Benjamin Dozier, \emph{Convergence of {S}iegel--{V}eech constants}, Geometriae
  Dedicata (2018).

\bibitem[EM01]{em2001}
Alex Eskin and Howard Masur, \emph{Asymptotic formulas on flat surfaces},
  Ergodic Theory and Dynamical Systems \textbf{21} (2001), no.~02, 443--478.

\bibitem[EMM98]{emm1998}
Alex Eskin, Gregory Margulis, and Shahar Mozes, \emph{Upper bounds and
  asymptotics in a quantitative version of the {O}ppenheim conjecture}, Ann. of
  Math. (2) \textbf{147} (1998), no.~1, 93--141. \MR{1609447}

\bibitem[EMM15]{emm2015}
Alex Eskin, Maryam Mirzakhani, and Amir Mohammadi, \emph{Isolation,
  equidistribution, and orbit closures for the {${\rm SL}(2,\mathbb{R})$}
  action on moduli space}, Ann. of Math. (2) \textbf{182} (2015), no.~2,
  673--721. \MR{3418528}

\bibitem[Esk06]{eskin2006}
Alex Eskin, \emph{Counting problems in moduli space}, Handbook of dynamical
  systems. {V}ol. 1{B}, Elsevier B. V., Amsterdam, 2006, pp.~581--595.
  \MR{2186249}

\bibitem[FK36]{fk1936}
Ralph~H. Fox and Richard~B. Kershner, \emph{Concerning the transitive
  properties of geodesics on a rational polyhedron}, Duke Math. J. \textbf{2}
  (1936), no.~1, 147--150. \MR{1545913}

\bibitem[KMS86]{kms1986}
Steven Kerckhoff, Howard Masur, and John Smillie, \emph{Ergodicity of billiard
  flows and quadratic differentials}, Ann. of Math. (2) \textbf{124} (1986),
  no.~2, 293--311. \MR{855297}

\bibitem[Mas88]{masur1988}
Howard Masur, \emph{Lower bounds for the number of saddle connections and
  closed trajectories of a quadratic differential}, Holomorphic functions and
  moduli, {V}ol.\ {I} ({B}erkeley, {CA}, 1986), Math. Sci. Res. Inst. Publ.,
  vol.~10, Springer, New York, 1988, pp.~215--228. \MR{955824}

\bibitem[Mas90]{masur1990}
\bysame, \emph{The growth rate of trajectories of a quadratic differential},
  Ergodic Theory and Dynamical Systems \textbf{10} (1990), no.~01, 151--176.

\bibitem[MTW16]{mtw2016}
L.~{Marchese}, R.~{Trevi{\~n}o}, and S.~{Weil}, \emph{{Diophantine
  approximations for translation surfaces and planar resonant sets}},
  arXiv:1502.05007v2 (2016).

\bibitem[{Nev}17]{nevo2017}
A.~{Nevo}, \emph{{Equidistribution in measure-preserving actions of semisimple
  groups : case of $SL_2(\mathbb{R})$}}, arXiv:1708.03886 (2017).

\bibitem[Vee89]{veech1989}
W.~A. Veech, \emph{Teichm\"uller curves in moduli space, {E}isenstein series
  and an application to triangular billiards}, Invent. Math. \textbf{97}
  (1989), no.~3, 553--583. \MR{1005006}

\bibitem[Vee98]{veech1998}
William~A. Veech, \emph{Siegel measures}, Ann. of Math. (2) \textbf{148}
  (1998), no.~3, 895--944. \MR{1670061}

\bibitem[Vor05]{vorobets2005}
Yaroslav Vorobets, \emph{Periodic geodesics on generic translation surfaces},
  Algebraic and topological dynamics, Contemp. Math., vol. 385, Amer. Math.
  Soc., Providence, RI, 2005, pp.~205--258. \MR{2180238}

\bibitem[Wri15]{wright2015}
Alex Wright, \emph{Translation surfaces and their orbit closures: an
  introduction for a broad audience}, EMS Surv. Math. Sci. \textbf{2} (2015),
  no.~1, 63--108. \MR{3354955}

\bibitem[ZK75]{kz1975}
A.~N. Zemljakov and A.~B. Katok, \emph{Topological transitivity of billiards in
  polygons}, Mat. Zametki \textbf{18} (1975), no.~2, 291--300. \MR{0399423}

\bibitem[Zor06]{zorich2006}
Anton Zorich, \emph{Flat surfaces}, Frontiers in number theory, physics, and
  geometry. {I}, Springer, Berlin, 2006, pp.~437--583. \MR{2261104}

\end{thebibliography}
\bibliographystyle{amsalpha}

\end{document}